\documentclass{article}

\usepackage{authblk}
\usepackage{graphicx}
\usepackage{amssymb}
\usepackage{amsthm}
\usepackage{amsmath}

\newtheorem{theorem}{Theorem}
\newtheorem{proposition}{Proposition}

\theoremstyle{definition}
\newtheorem{example}{Example}[section]

\theoremstyle{remark}
\newtheorem{remark}{Remark}

\DeclareMathOperator{\diag}{diag}
\DeclareMathOperator{\Ker}{Ker}

\begin{document}
\title{\textbf{Long term dynamics of the discrete growth-decay-fragmentation equation}}
\author[a]{J. Banasiak}
\author[*]{L.O. Joel}
\author[*]{S. Shindin}
\affil[a]{ Department of Mathematics and Applied Mathematics, University of
Pretoria, South Africa \& Institute of Mathematics, Technical University of \L\'{o}d\'{z}, \L\'{o}d\'{z}, Poland}
\affil[*]{School of Mathematics, Statistics and Computer Science\\
           University of Kwazulu-Natal, Durban,
           South Africa}

\date{}
\maketitle

\pagestyle{myheadings}
\markboth{J. Banasiak, L.O. Joel and S. Shindin}{Discrete growth-decay-fragmentation equation}
\thispagestyle{empty}
\begin{abstract}
In this paper, we prove that for a large class of growth-decay-fragmentation problems the solution semigroup is analytic and compact and thus has the Asynchronous Exponential Growth property.
\end{abstract}

\section{Introduction}\label{sec1}

Coagulation and fragmentation models that describe the processes of objects forming larger clusters or, conversely,
splitting into smaller fragments,  have received a lot of attention over several decades due to their importance in
chemical engineering and other fields of science and technology, see e.g. \cite{drake, ZM1}. One of the most
efficient approaches to modelling   dynamics of such processes is through the kinetic (rate)  equation which describes
the evolution of the distribution of interacting clusters with respect to their size/mass. The first model of this kind,
consisting of an infinite system of ordinary differential equations, was derived by Smoluchowski,
\cite{Smoluch, Smoluch17}, to describe  pure coagulation in the discrete case; that is,
if the ratio of the mass of the basic building block (monomer) to the mass of a typical cluster is positive,
and thus the size of a cluster is a finite multiple of the masses of the monomers. In many applications, however,
it turned out to be advantageous  to allow clusters to be composed of particles of any size $x>0$. This leads
to the continuous integro-differential equation  that was derived  by M\"{u}ller in the pure coagulation case,
\cite{muller1928allgemeinen}, and extended to a coagulation--fragmentation version in \cite{Melz53}.

In the last few decades it has been observed that also living organisms form clusters or split into subgroups depending on
circumstances, see e.g. \cite{gueron1995dynamics, Oku, okubo1986dynamical, DLP17} for modelling concerning larger
animals, or \cite{jackson1990model, Ackleh1} for phytoplankton models. It turns out that also the process of cell division
may be modelled within the same framework, see e.g. \cite{bell1967cell, sinko1971model, PerTr, BaPiRu}. What was not
always fully recognized in some papers mentioned above was that the living matter has its own vital dynamics; that is, in
addition to forming or breaking clusters, individuals within them are born or die  and so the latter processes must be
adequately represented in the models. In the continuous case, the birth and death processes are incorporated into the model
by adding an appropriate first order transport term, analogously to the age or size structured McKendrick model, see \cite{Ackleh1,
BaLa09, Bana12a, BaPiRu, PerTr}.  On the other hand, in the discrete case the vital processes are modelled by adding the
classical birth-and-death terms to the Smoluchowski equation. Note that e.g. the pure birth terms (or pure death terms) can
be obtained by the Euler discretization of the first order differential operator of the continuous case while the full birth-and-
death problem can be thought of as the discretization of the diffusion operator.

One of the most important problems in the analysis of dynamical systems is to determine their long term behaviour and
hence this aspect of the theory of growth--fragmentation equations has received much attention. The first systematic
mathematical study of the binary cell division model was carried out using semigroup theory in \cite{diekmann1984stability};
the semigroup approach was significantly extended to more general models in \cite{MiSc16}. Recently a number of results
have been obtained by the General Relative Entropy (or related) methods that lead to convergence of solutions in spaces
weighted by the eigenvector of the adjoint problem, see e.g. \cite{DoEs16, DoGa10, banasiakanalysis, LaPe09, MMP05,
PerTr, PeRy05}.

All above results concern growth--fragmentation models with continuous size distribution. Recently it has been observed,
however, that a large class of discrete fragmentation equations has much better properties than their continuous counterparts,
especially when considered in spaces where sufficiently high moments of solutions are finite. In particular, the fragmentation
operator in such spaces generates compact analytic semigroups. In this paper, we explore these ideas for the full growth-
death-fragmentation equation and show, in particular, that under natural assumptions on the coefficients of the problem the
growth-death-fragmentation semigroup is analytic, compact and irreducible and thus has the Asynchronous Exponential
Growth (AEG) property, see \cite{Arino}.

\section{The model}\label{sec2}
We consider a collection of clusters of sizes $n \in \mathbb N;$ that is, consisting of $n$ monomers
(cells, individuals,\ldots)  described by their size specific density $f=(f_n)_{n\in\mathbb{N}}$.
We assume that the number of monomers in each cluster can change by, say, a cell division (with the daughter cell
staying in the cluster) or its death. In an inanimate scenario, this can happen by the deposition of a particle from
the solute or, conversely, by its dissolution. If we assume that the probability of more than one  birth or death event in
a cluster happening simultaneously is negligible, then the process can be modelled by the classical birth-and-death system
of equations, see e.g. \cite[p. 1199]{jackson1990model}. We note that in the case of continuous size distribution the  growth process is modelled by the first
order differential operator with respect to size, $f \to -\partial_x(g f)$, where $g$ is the growth rate, see e.g. \cite{Bana12a},
whose Euler discretization with step-size 1 at $x =n$ is $g(n)f(n) - g(n+1)f(n+1).$ Similarly, the decay operator
$f \to \partial_x(d f)$ can be discretized as $-d(n)f(n) + d(n+1)f(n+1)$ and, using a central difference scheme,
the diffusion operator $f\to \partial_x(D\partial_xf)$ yields $D(n+1)f(n+1) - (D(n+1)+D(n))f(n) + D(n)f(n-1)$.

We further assume that the clusters can split into several smaller clusters. Combining both processes,
we arrive at the following system of equations:
\begin{align}
\frac{df_1}{dt}&= -g_1f_1 + d_2 f_2 + \sum\limits_{i=2}^\infty a_{i}b_{1,i}f_{i},\nonumber\\
\frac{df_n}{dt}&= g_{n-1}f_{n-1}-(a_n+g_n+d_n)f_n + d_{n+1}f_{n+1}
+ \sum\limits_{i=n+1}^\infty a_{i}b_{n,i}f_{i}, \;\; n\geq 2,\nonumber\\
f_n(0)&=f^{in}_n, \quad n\geq 1,
\label{feco1}
\end{align}
or
\begin{align}
\frac{df}{dt}&=\mathcal{G}^-f +(\mathcal{A}+\mathcal{G}^0+\mathcal{D}^0)f +\mathcal{D}^+f
+\Xi\mathcal{A}f= \mathcal{G}f+ \mathcal{D}f+ \mathcal{A}f+ \mathcal{B}f \nonumber\\
& = \mathcal{G}f+ \mathcal{D}f+ \mathcal{F}f\nonumber,\\
 f(0) &= f^{in}, \label{fecof}
\end{align}
where $f = (f_n)_{n\in\mathbb{N}}$ is the vector whose components $f_n$ give the numbers of $n$-clusters,
$\mathcal{A} = \diag(-a_n)_{n\geq 1}$, with $a_1=0$ and $a_n>0, n\geq 2,$  gives
the rates at which the clusters of mass $n$ undergo splitting,  $\mathcal{G}^0 = \diag(-g_n)_{n\geq 1}$,
$g_n\geq 0, n\geq 1,$ is the growth rate, $\mathcal{D}^0 = \diag(-d_n)_{n\geq 1}$,
$d_1=0$, $d_n\geq 0, n\geq 2,$ is the death rate, $\mathcal{G}^-, \mathcal{D}^+$ are, respectively, the
left and right shifts of $\mathcal{G}$ and $\mathcal{D}$; that is
\[
\mathcal{G}^-f = (0,g_1f_1,\ldots g_nf_n,\ldots),\quad \mathcal{D}^+f = (d_2f_2,\ldots d_nf_n,\ldots),
\]
$\Xi= (b_{n,i})_{1\le n < i,i\ge 2}$ is the daughter distribution function, also called the fragmentation
kernel, that gives the numbers of $i$-clusters resulting from splitting of a mass $n$ parent and
\[
\mathcal{G}=\mathcal{G}^-+\mathcal{G}^0, \quad
\mathcal{D} = \mathcal{D}^++\mathcal{D}^0, \quad
\mathcal{B} = \Xi\mathcal{A}, \quad
\mathcal{F} = \mathcal{A} +\mathcal{B}.
\]
Coefficients   $b_{n,i}$, $1\le n< i$, $i\ge 2$, are nonnegative numbers satisfying
\begin{equation}
\sum\limits_{n=1}^{i-1}nb_{n,i} = i.
\label{bin}
\end{equation}
The total mass of the ensemble is given by
\begin{equation}
M(t) = \sum\limits_{n=1}^\infty nf_n(t), \qquad t\geq 0;
\label{masscons1}
\end{equation}
then it is known, see e.g. \cite{Bana12b, BaLa12b}, that in the pure fragmentation case
($\mathcal{G}=\mathcal{D} =0$) the mass is conserved
\[
\frac{dM}{dt} = 0;
\]
that is,
\[
M(t) =  \sum\limits_{i=1}^\infty nf_n(0).
\]
Later, we shall use the fact that  (\ref{feco1}) can be written as the pure growth-fragmentation model
\begin{align}
\frac{df_1}{dt} &= -g_1f_1 + \sum\limits_{i=2}^\infty \mathsf{a}_{i} \mathsf{b}_{1,i}f_{i},\nonumber\\
\frac{df_n}{dt}&= g_{n-1}f_{n-1}-(g_n+\mathsf{a}_n)f_n + \sum\limits_{i=n+1}^\infty \mathsf{a}_{i}
\mathsf{b}_{n,i}f_{i},
\quad n\geq 2,\nonumber\\
f_n(0)&= f_n^{in}, \quad n\geq 1,
\label{feco2}
\end{align}
where $\mathsf{a}_n = a_n+d_n$, $n\geq 2,$ (with $\mathsf{a}_1=0$) and
\begin{equation}
\mathsf{b}_{n,i} = \left\{\begin{array}{lc}\frac{a_{n+1}b_{n,n+1} + d_{n+1}}{a_{n+1}+d_{n+1}},& i=n+1,\\
\frac{a_ib_{n,i}}{a_i+d_i},& i\geq n+2.
\end{array}\right.
\label{betani}
\end{equation}
We note that the fragmentation part of this model no longer is conservative as
\begin{equation}
\sum\limits_{n=1}^{i-1}n\mathsf{b}_{n,i} = i\left(1-\frac{d_i}{i(a_i+d_i)}\right), \qquad i\geq 2,
\label{betain}
\end{equation}
so it corresponds to the model with the so-called discrete mass-loss with mass-loss fraction
$\lambda_n = d_n/n(a_n+d_n)$, see \cite{Ed1,6}, mathematically analysed in \cite{SLLM12}.

The analysis of the pure fragmentation equation most often is carried out in the space $X_1:=\ell^1_1$ with the norm
 \begin{equation}
 \| f\|_{[1]} = \sum\limits_{n=1}^\infty n|f_n|
 \label{norm1}
 \end{equation}
 which, for a nonnegative $f,$ gives the mass of the ensemble. However,
 it is much better to consider (\ref{feco1}) in the  spaces with finite higher moments, $X_m:=\ell^1_m$, with the norm
 \begin{equation}
 \| f\|_{[m]} = \sum\limits_{n=1}^\infty n^m |f_n|, \qquad m\geq 1.
 \label{normp}
 \end{equation}
In the sequel, for any infinite diagonal matrix $\mathcal{P} = \diag(p_n)_{n\geq 1}$, we define the operator
$P_m$  in $X_m$ by $P_mf = \mathcal{P}f$ on $D(P_m) = \{f\in X_m;\; \mathcal P f\in X_m\}$.

\section{Analysis of the subdiagonal part}\label{sec3}
In this section, we shall consider the simplified problem corresponding to the subdiagonal part of (\ref{feco2}),
\begin{equation}
\frac{df}{dt} =\mathcal{K}f = \mathcal{G}^-f +(\mathcal{A}+\mathcal{G}^0+\mathcal{D}^0)f, \quad
f(0) = f^{in}. \label{feco3}
\end{equation}
Denote for brevity $\mathcal{T} = \mathcal{A}+\mathcal{G}^0+\mathcal{D}^0$ and consider the operator
$(T_m, D(T_m))$ defined, as above, by $T_m f = \mathcal{T}f$ on $D(T_m) = \{f \in X_m;\; \mathcal T f \in X_m\}$.
Then $G^-_m:= \mathcal{G}^-|_{D(T_m)}$ is a well defined positive operator in $X_m$ and we can apply
the substochastic semigroup theory, \cite{BaAr}, to $\mathcal{K}|_{D(T_m)} = T_m+G^-_m$.
Let $K_{m,\max}$ denote the maximal extension of $K_m$; that is, $K_{m,\max}f =
\mathcal{T}f + \mathcal{G}^-f$ on $$D(K_{m,\max}) =\{f \in X_m;\; \sum\limits_{n=2}^\infty n^m |a_nf_n +d_n f_n +g_n f_n-g_{n-1}f_{n-1}|<\infty\}.$$
\begin{theorem}
\begin{enumerate}
\item If
\begin{equation}
\liminf\limits_{n\to \infty} \left(a_n+d_n - g_n\frac{(n+1)^m-n^m}{n^{m}}\right) \geq 0,\label{condi1}
\end{equation}
then there is an extension ${K}_m$ of $T_m+G^-_m$ that generates a quasicontractive (of type $\mathcal{G}(1,\omega)$
for some $\omega\in \mathbb{R}$) positive semigroup on $X_m$ and, moreover, ${K}_m = K_{m,\max}$.
\item If there is $m'>m$ such that
\begin{equation}
\liminf\limits_{n\to \infty}\frac{n(a_n+d_n)}{g_n} \geq m',
\label{condi2}
\end{equation}
then (\ref{condi1}) is satisfied and the resolvent $R(\lambda, {K}_m)$ for $\lambda>\omega$ is given by
\begin{equation}
\label{lgf5}
[R(\lambda, {K}_m)f]_n = \sum_{i=1}^{n} \frac{f_i}{\lambda + \theta_n}
\prod_{j=i}^{n-1} \frac{g_{j}}{\lambda + \theta_{j}}, \quad n\ge 1,
\end{equation}
where $\theta_1 = g_1$ and $\theta_n = g_n+d_n+a_n, n\geq 1.$ Moreover, $D({K}_m)=D(A_m)\cap D(D^0_m)\cap D(G_m),$ where $G_m = \mathcal G|_{D(G_m)}$ with
\begin{equation}
D(G_m) =\{ f\in X_m;\; \sum\limits_{n=2}^\infty |g_{n}f_{n}-g_{n-1}f_{n-1}|<\infty\},
\label{dgm}
\end{equation}
and $({K}_m, D({K}_m)) = \overline{(T_m+G^-_m, D(T_m))}$.
\item If (\ref{condi2}) is satisfied, then $R(\lambda, {K}_m),$ $\lambda>\omega$, is compact provided
\begin{equation}
\liminf\limits_{n\to \infty}(a_n+d_n) = \infty
\label{condi3}
\end{equation}
and either
\begin{equation}
\sum\limits_{n=1}^\infty \frac{1}{g_n} <\infty, \quad \text{or}\quad
\lim\limits_{n\to \infty}\frac{n}{g_n} = 0.
\label{condi4}
\end{equation}
\item If
\begin{equation}
\liminf\limits_{n\to \infty}\frac{a_n+d_n}{g_n}> 0,
\label{riai}
\end{equation}
then $K_m =  G^-_m+T_m = G_m^-+G^0_m+D^0_m+ A_m$ and $(G_{K_m}(t))_{t\ge 0}$
is an analytic semigroup. If additionally (\ref{condi3}) is satisfied, then $(G_{K_m}(t))_{t\ge 0}$ is compact.
\end{enumerate}\label{mth1}
\end{theorem}
\begin{proof}
ad 1.) As in (\ref{feco2}), we denote $\mathsf{a}_n = a_n+d_n, n\geq 1$.
Let $\mathsf{a}_n- n^{-m}g_n((n+1)^m-n^m)\geq 0$ for $n\geq n_0$. Then for $f \in D(T_m)_+$ we have
\begin{align}
\sum\limits_{n=1}^\infty n^m[(T_m +G^{-}_m)f]_n =&
-\sum\limits_{n=2}^\infty n^m f_n \left(\mathsf{a}_n - g_n \frac{(n+1)^m-n^m}{n^m} \right)\nonumber \\
=&   -\left(\sum\limits_{n=2}^{n_0-1}+\sum\limits_{n=n_0}^{\infty}\right) n^m f_n \left(\mathsf{a}_n - g_n \frac{(n+1)^m-n^m}{n^m} \right)\nonumber\\
 =& c_0(f) - c_1(f),
\label{1est}
\end{align}
where $c_0$ is a bounded functional on $X_m$ and $c_1$ is nonnegative. Thus, as in  \cite[Proposition 9.29]{BaAr},
there is an extension ${K}_m\supset G_m^-+T_m$ generating a smallest quasicontractive (with the growth
rate $\omega$ not exceeding $\|c_0\|$) positive semigroup. By \cite[Theorem 6.20]{BaAr},
${K}_m\subset K_{m,\max}$. However, it is immediate that $\Ker (\lambda I - K_{m,\max})
= \{ 0\}$, hence \cite[Lemma 3.50 \& Proposition 3.52]{BaAr} gives ${K}_m= K_{m,\max}$.

\noindent
ad 2.) Since $c_1$ extends to $D({K}_m)_+$, for $f \in D({K}_m)_+$ we can write
\begin{align}
c_1(f) &= \lim\limits_{l \to \infty} \left(\sum\limits_{n=n_0}^{l} n^m f_n \left(\mathsf{a}_n - g_n \frac{(n+1)^m-n^m}{n^m} \right) -l^{m} g_l f_l\right)\nonumber\\
&= \sum\limits_{n=n_0}^{\infty} n^m f_n \left(\mathsf{a}_n - g_n \frac{(n+1)^m-n^m}{n^m} \right) -\lim\limits_{l \to \infty} l^{m}g_lf_l\label{rlfl}
\end{align}
and hence the last limit exists. Further, we have
\begin{align*}
c_1(f) &= \sum\limits_{n=n_0}^{\infty} n^m f_n \left(\mathsf{a}_n - g_n \frac{(n+1)^m-n^m}{n^m} \right) \\
&= \sum\limits_{n=n_0}^{\infty} n^m f_n\mathsf{a}_n \left(1 - \frac{g_n}{n \mathsf{a}_n}\left(m + O\left(\frac{1}{n}\right) \right)\right).
\end{align*}
If (\ref{condi2}) is satisfied, then (possibly adjusting $n_0$ from the previous part of the proof) for  $n\geq n_0$
\[
 1- \frac{g_n}{n \mathsf{a}_n}\left(m + O\left(\frac{1}{n}\right) \right) \geq  1-\frac{m}{m'} +
  \frac{g_n}{n\mathsf{a}_n}O\left(\frac{1}{n}\right) \geq c'>0
\]
on account of $m'>m$ and $g_n/n\mathsf{a}_n \leq 1/m'$. Since $c_1$ extends to $D({K}_m)_+$
by monotonic limits,
we argue as in \cite[Theorem 2.1]{Bana12b} that any $f\in D({K}_m)$ is summable with the weights
$(n^m\mathsf{a}_n)_{n\geq 1}$ and hence, by (\ref{condi2}), it is also summable with the weight $(n^{m-1}{g}_n)_{n\geq 1}$. Therefore, in particular,  $D({K}_m) \subset D(A_m)\cap D(D^0_m)$ and hence also $D({K}_m) \subset D(G_m)$ holds by the definition of $D(K_{\max,m})$. The converse inclusion is obvious. Further, from \eqref{rlfl} we know that $\lim_{l \to \infty} l^{m}g_lf_l$ exists, and thus it must be 0. Indeed, otherwise  $l^{m}g_lf_l > c$ for some $c>0$ and large $l$ contradicting the summability of  $(n^{m-1}{g}_n)_{n\geq 1}$. But then (\ref{rlfl}) implies that ${K_m}$ is honest, hence $({K}_m, D({K}_m)) = \overline{(T_m+G^-_m, D(T_m))}$ by \cite[Corollary 6.14]{BaAr}.

Let $\lambda>\omega$. We use the formula
\begin{equation}
R(\lambda, {K}_m) f = \sum\limits_{k=0}^\infty R(\lambda, T_m)[G_m^-R(\lambda, T_m)]^k  f,
\quad   f\in X_m,\quad \lambda >\omega_0,
\label{res1}
\end{equation}
\cite[Proposition 9.29]{BaAr}. Since $R(\lambda, T_m)$ is represented by the matrix
$\mathcal{R}(\lambda) = \diag\left (\frac{1}{\lambda+\theta_n}\right)_{n\geq 1}$, and
$G_m^-$ is represented by $\mathcal{G}^-$, we have
$\mathcal{R}(\lambda)[\mathcal{G}^-\mathcal{R}(\lambda)]^k
=  (\gamma^{(k)}_{ij})_{i,j \in \mathbb{N}}$, where
\[
\gamma^{(k)}_{ij} = \left\{
\begin{array}{lc}\frac{1}{\lambda+\theta_j}\prod\limits_{l =i}^{j-1}\frac{g_l}{\lambda +\theta_{l}},&
i=1,2,\ldots,\;\; j = i+k,\\
0,&\mathrm{otherwise}.
\end{array}\right.
\]
Since the convergence in $X_m$ implies the coordinate-wise convergence, we see that for each $n$ the component
$[R(\lambda, {K}_m)f]_n$ of the series (\ref{res1}) terminates after $n$ terms and hence the
resolvent is given by (\ref{lgf5}).

Though not strictly necessary, the estimates of the norm of the resolvent are instructive and used also further down. To simplify the calculations, instead of $\|\cdot\|_{[m]}$,  we employ the norm
$\|f\|_* := \sum_{n=1}^{\infty} \frac{\Gamma(n+m)}{\Gamma(n)} |f_n|$ that is equivalent to
$\|\cdot\|_{[m]}$ by virtue of the Stirling formula, e.g. \cite[formula 6.1.47]{AbrSte},
\begin{equation}
\frac{\Gamma(n+m)}{\Gamma(n)} = O(n^{m}).
\label{Stfor}
\end{equation}

Let $f \in X_m$ and $\lambda >\omega$. Then, changing the order of summation,
\begin{align}
\| R(\lambda, {K}_m)f \|_* & \le \sum_{i=1}^{\infty} |f_{i}| \sum_{n=i}^{\infty} \frac{\Gamma(n+m)}{\Gamma(n)}
\frac{1}{\lambda + \theta_n} \prod_{j=i}^{n-1} \frac{g_{j}}{\lambda + \theta_{j}}\nonumber\\
& =  \frac{1}{\lambda} \sum_{i=1}^{\infty} |f_{i}| \sum_{n=i}^{\infty} \frac{\Gamma(n+m)}{\Gamma(n)}
\bigg ( \frac{\lambda+g_n}{\lambda + \theta_n} - \frac{g_n}{\lambda + \theta_n} \bigg )
\prod_{j=i}^{n-1} \frac{g_j}{\lambda +\theta_j} \nonumber\\
&\leq \frac{1}{\lambda} \sum_{i=1}^{\infty} |f_{i}|  \sum_{n=i}^{\infty} \frac{\Gamma(n+m)}{\Gamma(n)}
\left( \prod_{j=i}^{n-1} \frac{g_j}{\lambda + \theta_{j}} - \prod_{j=i}^{n} \frac{g_j}{\lambda+ \theta_{j}}\right).
\label{eq14}
\end{align}
Now, we have
\begin{align}
&\sum_{n=i}^{\infty} \frac{\Gamma(n+m)}{\Gamma(n)}\left( \prod_{j=i}^{n-1}
\frac{g_j}{\lambda +\theta_{j}} - \prod_{j=i}^{n} \frac{g_j}{\lambda+ \theta_{j}}\right)\nonumber \\
&\qquad= \lim\limits_{N\to \infty}   \sum_{n=i}^{N} \frac{\Gamma(n+m)}{\Gamma(n)}\left( \prod_{j=i}^{n-1}
\frac{g_j}{\lambda + \theta_{j}} - \prod_{j=i}^{n} \frac{g_j}{\lambda+ \theta_{j}}\right)\nonumber \\
&\qquad =\frac{\Gamma(i+m)}{\Gamma(i)}\nonumber\\
&\qquad\phantom{xx} +\lim\limits_{N\to \infty}\left( \sum\limits_{n=i+1}^N  \frac{\Gamma(n+m)}{\Gamma(n)}
\prod_{j=i}^{n-1} \frac{g_j}{\lambda + \theta_{j}} - \sum_{n=i}^{N} \frac{\Gamma(n+m)}{\Gamma(n)}
\prod_{j=i}^{n} \frac{g_j}{\lambda+ \theta_{j}}\right)\nonumber\\
&\qquad =\frac{\Gamma(i+m)}{\Gamma(i)} +\lim\limits_{N\to \infty}\left( \sum\limits_{n=i+1}^N
\frac{\Gamma(n+m-1)}{\Gamma(n-1)} \prod_{j=i}^{n-1}
\frac{g_j}{\lambda + \theta_{j}} \right.\nonumber\\
&\qquad\phantom{xx}+\left. m \sum\limits_{n=i+1}^N  \frac{\Gamma(n+m-1)}{\Gamma(n)} \prod_{j=i}^{n-1}
\frac{g_j}{\lambda + \theta_{j}}- \sum_{n=i}^{N} \frac{\Gamma(n+m)}{\Gamma(n)}
\prod_{j=i}^{n} \frac{g_j}{\lambda+ \theta_{j}}\right) \nonumber\\
&\qquad=\frac{\Gamma(i+m)}{\Gamma(i)}\label{15}\\
&\qquad\phantom{xx} +\lim\limits_{N\to \infty}\left(  m \sum\limits_{n=i+1}^N  \frac{\Gamma(n+m-1)}{\Gamma(n)}
\prod_{j=i}^{n-1} \frac{g_j}{\lambda +\theta_{j}}- \frac{\Gamma(N+m)}{\Gamma(N)} \prod_{j=i}^{N}
\frac{g_j}{\lambda+ \theta_{j}}\right).\nonumber
\end{align}
Using (\ref{condi2}), for sufficiently large $j$ we have
\begin{equation}
\frac{g_j}{\lambda+ \theta_{j}} = \frac{1}{\frac{\lambda}{g_j}+ 1+ \frac{a_j+d_j}{g_j}} \leq \frac{j}{m' + j},
\label{prodest}
\end{equation}
hence
\begin{align*}
0&\leq \limsup\limits_{N\to \infty}\frac{\Gamma(N+m)}{\Gamma(N)} \prod_{j=i}^{N}
\frac{g_j}{\lambda+ \theta_{j}} \leq \limsup\limits_{N\to \infty}\frac{\Gamma(N+m)}{\Gamma(N)}
\frac{\Gamma(N+1)\Gamma(i+m')}{\Gamma(N+m'+1)\Gamma(i)}\\
&= \limsup\limits_{N\to \infty} \frac{N}{N+m'}\frac{\Gamma(N+m)\Gamma(i+m')}{\Gamma(N+m')\Gamma(i)} = 0
\end{align*}
on account of the Stirling formula, see (\ref{Stfor}). Thus, (\ref{15}) can be continued as
\begin{align}
& \frac{\Gamma(i+m)}{\Gamma(i)} + m \sum\limits_{n=i+1}^\infty \frac{\Gamma(n+m-1)}{\Gamma(n)}
\prod_{j=i}^{n-1} \frac{j}{j +m'} \nonumber\\
&\qquad\qquad= \frac{\Gamma(i+m)}{\Gamma(i)} + m \sum\limits_{n=i+1}^\infty
\frac{\Gamma(n+m-1)\Gamma(i+m')}{\Gamma(i)\Gamma (n+m')}\nonumber \\
&\qquad\qquad= \frac{\Gamma(i+m)}{\Gamma(i)} + m \sum\limits_{n=i+1}^\infty
\frac{\Gamma(n+m-1)}{\Gamma(i)\Gamma(n-i)}B(n-i,m'+i),
\label{17}
\end{align}
where $B$ is the Beta function. The sum above can be computed explicitly. Indeed,
using the integral representation for the Beta function, we obtain
\begin{align*}
&\sum_{n=i+1}^{\infty} \frac{\Gamma(n+m-1)}{\Gamma(i) \Gamma(n-i)}
\int_{0}^{1} (1-t)^{m' + i - 1} t^{n-i-1} dt \\
&\qquad\qquad= \frac{1}{\Gamma(i)} \int_{0}^{1}\bigg( (1-t)^{m' + i - 1}   \sum_{n=i+1}^{\infty}
\frac{\Gamma(n+m-1)}{\Gamma(n)} \frac{d^{i}}{dt^{i}} t^{n-1} \bigg) dt  \\
&\qquad\qquad= \frac{1}{\Gamma(i)} \int_{0}^{1}\bigg( (1-t)^{m' + i - 1}  \frac{d^{i}}{dt^{i}}
\sum_{n=0}^{\infty}  \frac{\Gamma(n+m)}{\Gamma(n+1)} t^{n} \bigg) dt\\
&\qquad\qquad= \frac{\Gamma(m)}{\Gamma(i)} \int_{0}^{1} (1-t)^{m' + i - 1}
\bigg( \frac{d^{i}}{dt^{i}} \frac{1}{(1-t)^m} \bigg) dt\\
&\qquad\qquad=\frac{\Gamma(m+i)}{\Gamma(i)} \int_{0}^{1} (1-t)^{m' - m - 1}  dt
= \frac{\Gamma(m+i)}{\Gamma(i)}\frac{1 }{m' - m}.
\end{align*}
Substituting the above into (\ref{17}) and returning to (\ref{eq14}), we obtain
\begin{equation}
\|R(\lambda, {K}_m)f\|_* \leq \frac{m'}{m' - m}\frac{1}{\lambda} \|f\|_*.
\label{resest1}
\end{equation}
\noindent
ad 3.) To prove the compactness, we consider the projections
\begin{equation}
P_N f = (f_1,f_2,\ldots,f_N,0,\ldots), \quad N\geq 1.\label{proj}
\end{equation}
Since $P_NR(\lambda, {K}_m)$ is an operator with finite dimensional range, it is compact. We consider
\begin{align}
&\|P_{N-1}R(\lambda, {K}_m) f - R(\lambda, {K}_m) f\|_*
\leq\sum\limits_{n=N}^\infty \frac{\Gamma(n+m)}{\Gamma(n)} \sum\limits_{i=1}^{n-1}
\frac{|f_i|}{\lambda+\theta_n}\prod\limits_{j=i}^{n-1}\frac{g_j}{\lambda +\theta_j}\nonumber\\
&= \sum\limits_{i=1}^{N-1}|f_i|S_{N,i}+\sum\limits_{i=N}^{\infty}|f_i|S_{i+1,i},\label{SN}
\end{align}
where
\[
S_{l,i} = \sum\limits_{n=l}^\infty \frac{\Gamma(n+m)}{\Gamma(n)}
\frac{1}{\lambda+\theta_n}\prod\limits_{j=i}^{n-1}\frac{g_j}{\lambda+\theta_j}.
\]
Now,
\begin{align*}
S_{i+1,i} &= \sum\limits_{n=i+1}^\infty \frac{\Gamma(n+m)}{\Gamma(n)}
\frac{1}{\lambda+\theta_n}\prod\limits_{j=i}^{n-1}\frac{g_j}{\lambda+\theta_j}\\
&\leq  \sup\limits_{n\geq i+1}\left\{\frac{1}{\mathsf{a}_n}\right\}\sum\limits_{n=i}^\infty
\frac{\Gamma(n+m)}{\Gamma(n)} \frac{\mathsf{a}_n}{\lambda+\theta_n}
\prod\limits_{j=i}^{n-1}\frac{g_j}{\lambda+\theta_j}\\
&\leq  \sup\limits_{n\geq i+1}\left\{\frac{1}{\mathsf{a}_n}\right\}\sum\limits_{n=i}^\infty
\frac{\Gamma(n+m)}{\Gamma(n)} \left(1-\frac{g_n}{\lambda+\theta_n}\right)
\prod\limits_{j=i}^{n-1}\frac{g_j}{\lambda+\theta_j}\\
&=  \sup\limits_{n\geq i+1}\left\{\frac{1}{\mathsf{a}_n}\right\}\sum\limits_{n=i}^\infty
\frac{\Gamma(n+m)}{\Gamma(n)} \left(\prod\limits_{j=i}^{n-1}
\frac{g_j}{\lambda+\theta_j}-\prod\limits_{j=i}^{n}\frac{g_j}{\lambda+\theta_j}\right)\\
&\leq \sup\limits_{n\geq i+1}\left\{\frac{1}{\mathsf{a}_n}\right\}\frac{m'}{m'-m}\frac{\Gamma(i+m)}{\Gamma(i)},
\end{align*}
where we used the estimates for (\ref{eq14}). Hence
\begin{eqnarray}
\sum\limits_{i=N}^{\infty}|f_i|S_{i+1,i}&\leq& \sup\limits_{n\geq N+1}\left\{\frac{1}{\mathsf{a}_n}\right\}
\frac{m'}{m'-m}\sum\limits_{i=N}^{\infty}|f_i|\frac{\Gamma(i+m)}{\Gamma(i)}\nonumber\\
 &\leq& \sup\limits_{n\geq N+1}\left\{\frac{1}{\mathsf{a}_n}\right\}\frac{m'}{m'-m}\|f\|_*
\label{1ste}
\end{eqnarray}
and, by (\ref{condi3}), this term tends to 0 as $N\to \infty$, uniformly on the unit ball of $X_m$.

Since, by (\ref{prodest}),
\begin{align*}
\prod\limits_{j=i}^{n}\frac{g_j}{\lambda+\theta_j} &\leq  \frac{i}{m'+i}\cdot\ldots\cdot \frac{n}{m'+n}
\leq \frac{i}{m+i}\cdot\ldots\cdot \frac{n}{m+n} =\frac{\Gamma(n+1)\Gamma(i+m)}{\Gamma(i)\Gamma(n+m+1)}\\
&= \frac{\Gamma(n)\Gamma(i+m)}{\Gamma(i)\Gamma(n+m)}\frac{n}{n+m},
\end{align*}
we have
\begin{align*}
S_{N,i} &= \sum\limits_{n=N}^\infty \frac{\Gamma(n+m)}{\Gamma(n)}
\frac{1}{\lambda+\theta_n}\prod\limits_{j=i}^{n-1}\frac{g_j}{\lambda+\theta_j}
= \sum\limits_{n=N}^\infty \frac{1}{g_n}\frac{\Gamma(n+m)}{\Gamma(n)}
\prod\limits_{j=i}^{n}\frac{g_j}{\lambda+\theta_j}\\
&\leq \frac{\Gamma(i+m)}{\Gamma(i)} \sum\limits_{n=N}^\infty \frac{1}{g_n}\frac{n}{n+m}
\leq \frac{\Gamma(i+m)}{\Gamma(i)} \sum\limits_{n=N}^\infty \frac{1}{g_n}.
\end{align*}
Hence
\[
\sum\limits_{i=1}^{N-1}|f_i|S_{N,i} \leq \left(\sum\limits_{n=N}^\infty \frac{1}{g_n}\right)
\sum\limits_{i=1}^{\infty}|f_i|\frac{\Gamma(i+m)}{\Gamma(i)}
\]
and, using the first option of (\ref{condi4}) and combining the above estimate with (\ref{1ste}), we see that
\[
\lim\limits_{N\to \infty}P_{N-1}R(\lambda, {K}_m) =  R(\lambda, {K}_m)
\]
in the uniform operator norm. Therefore $R(\lambda, {K}_m)$ is compact.

To use the second option of (\ref{condi4}), first we re-write the formula for $S_{N,i}$ as
\begin{align*}
S_{N,i} &= \sum\limits_{n=N}^\infty \frac{\Gamma(n+m)}{\Gamma(n)}
\frac{1}{\lambda+\theta_n}\prod\limits_{j=i}^{n-1}\frac{g_j}{\lambda+\theta_j}
= \sum\limits_{n=N}^\infty \frac{1}{g_n}\frac{\Gamma(n+m)}{\Gamma(n)} \prod\limits_{j=i}^{n}
\frac{g_j}{\lambda+\theta_j}\\
&\leq \frac{\Gamma(i+m')}{\Gamma(i)} \sum\limits_{n=N}^\infty \frac{1}{g_n}
\frac{\Gamma(n+m)}{\Gamma(n+m')}\frac{n}{n+m'}.
\end{align*}
Then, using  again the Stirling formula, for large $i$ and $N>i$ we can write
\[
\frac{\Gamma(i+m')}{\Gamma(i)} \sum\limits_{n=N}^\infty \frac{1}{g_n}
\frac{\Gamma(n+m)}{\Gamma(n+m')}\frac{n}{n+m'}\leq C i^mN^{m'-m}
\sum\limits_{n=N}^\infty \frac{1}{n^{m'-m}g_n}
\]
for some constant $C$, since $m'-m>0$. Now, by assumption, $g^{-1}_nn^{m-m'}$ is summable
(as $g_n = O(n)$ and $m-m'<0$), hence $\sum\limits_{n=N}^\infty \frac{1}{n^{m'-m}g_n}$ converges to $0$
as $N\to \infty$. Since $N^{m-m'}$ monotonically converges to $0$, we can use the Stolz--Ces\'{a}ro theorem.
We have
\begin{align*}
\lim\limits_{N\to \infty} \frac{\frac{1}{g_N N^{m'-m}}}{N^{m-m'} - (N+1)^{m-m'}}
&=\frac{1}{m'-m}\lim\limits_{N\to \infty} \frac{N+1}{g_N} \\
&= \lim\limits_{N\to \infty} N^{m'-m} \sum\limits_{n=N}^\infty \frac{1}{n^{m'-m}g_n}.
\end{align*}
Since, by our assumption, the limit in the middle equals 0, we see that
\[
\sum\limits_{i=1}^{N-1}|f_i|S_{N,i} \leq \left(N^{m'-m}
\sum\limits_{n=N}^\infty \frac{1}{n^{m'-m}g_n}\right)\sum\limits_{i=1}^{\infty}|f_i|i^m
\]
and the thesis follows as above.

ad 4.) By (\ref{riai}), $g_n \leq C(a_n+d_n)$ for large $n$ and some $C>0,$ hence (\ref{condi2}) holds and
thus also the thesis of 2. holds. Moreover, (\ref{riai}) implies
\[
D(A_m)\cap D(D_m^0) \subset D(T_m)
\]
and hence, by 2.), $D({K}_m)\subset D(T_m)$. Since ${K}_m$ is an extension of
$(T_m+G^-_m, D(T_m))$, we see that $K_m = T_m+G^-_m$, but then we also have
$K_m = A_m+D^0_m+G^0_m+G^-_m$. Further, since $(T_m, D(T_m))$ is a diagonal operator, it generates an
analytic semigroup and hence $(K_m, D(T_m))$ also generates an analytic semigroup by the Arendt--Rhandi
theorem, \cite{AR}.

Now, the stronger assumption on $g_n$ allows for a simpler proof of the compactness without the need for (\ref{condi4}).
By virtue of the above and \cite[Theorem 4.3]{BaAr},  $I  - G_m^-R(\lambda, T_m)$ is invertible and
\[
R(\lambda, T_m+G^-_m) = R(\lambda, T_m) [I - G_m^-R(\lambda, T_m) ]^{-1}.
\]
In view of the last identity, it suffices to show that $R(\lambda, T_m)$ is compact for some $\lambda>0$.
For each $f\in X_m$ with $\| f\|_{[m]}\le 1$, we have $\|R(\lambda, T_m) f\|_{[m]}\le {1}/{\lambda}$ and
\[
\sum_{n = n_0}^\infty n^m\bigl|[R(\lambda,T_m)  f ]_{n}\bigr|
\le \sup\limits_{n\ge n_0} \frac{1}{\lambda +\theta_n} \sum_{n = n_0}^\infty n^m|f_n|
\le \sup\limits_{n\ge n_0} \frac{1}{\lambda +\mathsf{a}_n} \sum_{n = n_0}^\infty n^m|f_n|.\]
If \eqref{condi3} holds, we have
\[
\lim_{n_0\to\infty} \sup_{n\ge n_0} \frac{1}{\lambda +\mathsf{a}_n}
= \frac{1}{\lambda +\liminf_{n\to \infty}\mathsf{a}_n} = 0.
\]
Hence the image of the unit ball $B=\{ f\in X : \|f\|_{[m]} \le 1\}$ under $R(\lambda, T_m)$ is bounded and
uniformly summable  and therefore it is precompact, see \cite[IV.13.3]{DS}.
Hence $R(\lambda, T_m)$ is compact and the compactness of $(G_{K_m}(t))_{t\ge0}$ follows
from \cite[Theorem II.4.29]{EN}.
\end{proof}
\begin{remark} We note that (\ref{riai}) also allows to apply the Miyadera perturbation theorem, see e.g. \cite[Theorem III.3.16]{EN}. Indeed, if (\ref{riai}) is satisfied, then we can find $n_0$ such that for $n\geq n_0+1$
$$
\frac{(n+1)^mg_n}{n^m (g_n+\mathsf{a}_n)}\leq q<1
$$
and then $\omega>0$ such that
$$
\max\limits_{1\leq n\leq n_0}\frac{(n+1)^mg_n}{n^m (g_n+\mathsf{a}_n +\omega)}\leq q<1.
$$
Since the generation for $T_m+G^-_m$ is equivalent to that for  $T_m+G^-_m-\omega I,$ \cite[Lemma 4.15]{BaAr},  the Miyadera condition for $T_m+G^-_m-\omega I$ and $f \in D(T_m)_+$ reads
\begin{align*}
\int_0^\delta \|G^-_mG_{T_m}(t)f\|_{[m]}dt &= \sum\limits_{n=1}^\infty \frac{(n+1)^mg_{n}(1-e^{-(g_n+\mathsf{a}_n+\omega)\delta})}{n^m (g_n+\mathsf{a}_n +\omega)}n^m f_n\\
&\leq \sum\limits_{n=1}^{n_0} \frac{(n+1)^mg_{n}}{n^m (g_n+\mathsf{a}_n +\omega)}n^m f_n + \sum\limits_{n=n_0+1}^{\infty} \frac{(n+1)^mg_{n}}{n^m (g_n+\mathsf{a}_n)}n^m f_n\\
&\leq q\|f\|_{[m]}.
\end{align*}

At the same time, if $\mathsf{a}_n/g_n\to 0$ as $n\to \infty$, then $g_n/( \mathsf{a}_n+g_n) \to 1$ and the above estimate is not available.
\end{remark}
\section{Growth-fragmentation equation}\label{sec4}
We introduce the following notation, see \cite{Bana12b},
\begin{equation}
\Delta_n^{(m)} := n^m-b^{(m)}_n:=
n^m-\sum\limits_{k=1}^{n-1}k^mb_{k,n}, \quad
n\geq 2,\quad m\geq 0. \label{mpn}
\end{equation}
Then, for $n\geq 2,$
 \begin{equation} \Delta^{(0)}_n = 1-b^0_1\leq 0, \quad \Delta^{(1)}_n=0,
 \quad \Delta_n^{(m)} \geq0, \qquad m> 1. \label{Nm}
\end{equation}
Further, let us recall the notation $\theta_1=g_1$ and $\theta_n=a_n+g_n+d_n$, $n\geq 1$.
\begin{theorem}\label{theorem2}\begin{enumerate}
\item Let \eqref{condi2} be satisfied. If  for some $m>1$
 \begin{equation}
\liminf\limits_{n\to \infty}\frac{a_n}{\mathsf{a}_n}\frac{\Delta_n^{(m)}}{n^m}>\frac{m}{m'}
\label{crucrit'}
\end{equation}
holds, where  as before $\mathsf{a}_n =a_n+d_n$, then
\begin{align}
({Y}_m, D({K}_m)) &= ({K}_m + D_m^++ B_m, D({K}_m))\nonumber\\&= \overline{(T_m+G_m^-+D^+_m+B_m, D(T_m))}\label{thesame0}
\end{align}
 generates a positive semigroup in $X_m$. If additionally \eqref{condi3} and \eqref{condi4} are satisfied, $R(\lambda, {Y}_m)$ is compact for sufficiently large $\lambda$.
\item If  for some $m>1$
 \begin{equation}
\liminf\limits_{n\to \infty}\frac{a_n}{\theta_n}\frac{\Delta_n^{(m)}}{n^m}>0
\label{crucrit}
\end{equation}
holds, then
\begin{equation}
(U_m, D(T_m)) = (A_m+G_m+D_m + B_m, D(T_m)) = (Y_m, D(Y_m)),\label{thesame}
\end{equation}
where $D(T_m) = D(A_m+G_m^0+D_m^0),$ generates a positive, analytic semigroup in $X_k$ for any $k\geq m$.
\end{enumerate}
\end{theorem}
\begin{proof}
ad 1.) Repeating the calculations in (\ref{1est}) for the
full operator using $f\in D({K}_m)_+  \subset D(A_m)_+\cap  D(D^0_m)_+$ and \eqref{dgm}, we obtain
$$
\sum\limits_{n=1}^\infty n^m[({K}_m+D^+_m + B_m)f]_n = \sum\limits_{n=1}^\infty n^m[A_m+D_m + B_m)f]_n  + \sum\limits_{n=1}^\infty n^m[G_m f]_n
$$
Now, using the convention that $g_0f_0=0$,
\begin{align*}
\sum\limits_{n=1}^\infty n^m[G_m f]_n &= \lim\limits_{l\to \infty}\sum\limits_{n=1}^l n^m(g_{n-1}f_{n-1}-g_{n}f_{n}) \\
&= \lim\limits_{l\to \infty}(\sum\limits_{n=1}^l ((n+1)^m-n^m)g_{n}f_{n} - (l+1)^m g_{l+1}f_{l+1}) \\&=  \sum\limits_{n=1}^\infty ((n+1)^m-n^m)g_{n}f_{n}
\end{align*}
by the proof of Theorem \ref{mth1}, part 2. Hence
\begin{align}
&\sum\limits_{n=1}^\infty n^m[({K}_m+D^+_m + B_m)f]_n = \sum\limits_{n=1}^\infty
\mathsf{a}_n n^m f_n\left(\left(\left(1+\frac{1}{n}\right)^m-1\right)\frac{g_n}{\mathsf{a}_n}\right.\nonumber\\
&\qquad\qquad\left.-\frac{a_n}{\mathsf{a}_n} \left(1 - \frac{1}{n^m}\sum\limits_{k=1}^{n-1}k^mb_{k,n}\right)
-\frac{d_n}{\mathsf{a}_n}\left(1-\left(1-\frac{1}{n}\right)^m\right)\right)\nonumber\\
&\qquad\qquad=:- \sum\limits_{n=1}^\infty \Lambda_n \mathsf{a}_nn^m f_n.
\label{est2}
\end{align}
Thus, if $\Lambda_n \geq 0$ for large $n$, then there is an extension $({Y}_m, D({Y}_m))$ of $({K}_m +D_m^++B_m, D(A_m)\cap D(D^0_m)\cap D(G_m))$ generating a positive semigroup. Since
\[
\Lambda_n = \frac{a_n}{\mathsf{a}_n}\frac{\Delta_n^{(m)}}{n^m}+\frac{d_n}{\mathsf{a}_n}O\left(\frac{1}{n}\right)
-\frac{g_n}{n\mathsf{a}_n}\left(m +O\left(\frac{1}{n}\right)\right),
\]
where   both $\frac{d_n}{\mathsf{a}_n}O\left(\frac{1}{n}\right)$ and $\frac{g_n}{n\mathsf{a}_n}O\left(\frac{1}{n}\right)$
converge to zero due to the boundedness of $d_n/\mathsf{a}_n$ and $\frac{g_n}{n\mathsf{a}_n} \leq  m/m'$, $\Lambda_n \geq \Lambda> 0$ for some $\Lambda$ and  large $n$ if and only if  (\ref{crucrit'}) is satisfied. We observe that $\Lambda\leq 1$ so that $m<m'$ is a necessary condition for (\ref{crucrit'}) to hold. Hence, if (\ref{crucrit'}) holds, $D({Y}_m) \subset D(A_m)\cap D(D_m^0)$. Then, since $D(B_m), D(D_m^+)\subset D(A_m)\cap D(D_m^0)$ and ${Y}_m$ is a restriction of the maximal operator, $D({Y}_m)\subset D(G_m)$ and hence the first part of (\ref{thesame0}) is proved. To prove the second part, we note that $(K_m, D(K_m)) = \overline{(T_m+ G_m^-, D(T_m))}$. Since $K_m + D^+_m + B_m$ is the generator, it is closed and thus
\begin{align}
\overline{(T_m+ G_m^-+D^+_m +B_m, D(T_m))}&\subset \overline{(K_m + D^+_m + B_m, D(K_m))}\nonumber \\& = (K_m + D^+_m + B_m, D(K_m)).\label{close}
\end{align}
On the other hand,  $D(K_m)\subset D(D_m^0)\cap D(A_m)=D(D_m^+)\cap D(B_m)$, hence  $D^+_m +B_m$ is $K_m$-bounded by \cite[Lemma 4.1 \& Theorem 2.65]{BaAr}.
Let $f \in D(K_m)$. Then $f = \lim_{n\to\infty} f_n$ with $f_n \in D(T_m)$ and $\lim_{n\to\infty}K_mf_n = \lim_{n\to\infty}(T_m + G^-_m)f_n = K_mf$. By $K_m$-boundedness, $((D^+_m + B_m)f_n)_{n\in \mathbb N}$ converges. By \eqref{close}, $T_m+ G_m^-+D^+_m +B_m$ is closable and hence $$
K_m f +D^+_mf +B_mf = \lim_{n\to \infty} (T_m+ G_m^-+D^+_m +B_m)f_n = \overline{(T_m+ G_m^-+D^+_m +B_m)}f.$$ Thus
$$
K_m  + D^+_m +B_m \subset \overline{T_m+ G_m^-+D^+_m +B_m}
$$
and (\ref{thesame0}) follows.

The compactness of $R(\lambda, {Y}_m)$ follows from
\[
R(\lambda, {Y}_m) = R(\lambda, {K}_m) [I - (B_m + D_m^+)R(\lambda, {K}_m) ]^{-1},
\]
where the second term on the right hand side is a bounded operator by $D(B_m+D^+_m) = D(A_m)\cap D(D^0_m)\supset D({K}_m).$ Thus the proof of the compactness of $R(\lambda, {Y}_m)$ follows as in item 4.) of Theorem \ref{mth1}.

ad 2.)  By \cite[Theorem 2.1]{Bana12b}, if (\ref{crucrit}) holds for some $m_0,$ then it holds for any $m\geq 0$.
Hence, we can fix an $m$ for which (\ref{crucrit}) holds. Then, for $f\in D(T_m) = D(A_m+G_m^0+D_m^0) = D(A_m)\cap D(G_m^0)\cap D(D^0_m)$, we obtain
\begin{align}
&\sum\limits_{n=1}^\infty n^m[(G_m+D_m+ A_m + B_m)f]_n =  \sum\limits_{n=1}^\infty
\theta_n n^m f_n\left(\left(\left(1+\frac{1}{n}\right)^m-1\right)\frac{g_n}{\theta_n}\right.\nonumber\\
&\qquad\qquad\left.-\frac{a_n}{\theta_n} \left(1 - \frac{1}{n^m}\sum\limits_{k=1}^{n-1}k^mb_{k,n}\right)
-\frac{d_n}{\theta_n}\left(1-\left(1-\frac{1}{n}\right)^m\right)\right)\nonumber\\
&\qquad\qquad=:- \sum\limits_{n=1}^\infty \theta_n n^m f_n\Theta_n.
\label{est2}
\end{align}
Then we proceed as above. Since
\[
\Theta_n = \frac{a_n}{\theta_n}\frac{\Delta_n^{(m)}}{n^m}+\frac{d_n}{\theta_n}O\left(\frac{1}{n}\right)
-\frac{g_n}{\theta_n}O\left(\frac{1}{n}\right),
\]
where the terms $\frac{g_n}{\theta_n}O\left(\frac{1}{n}\right)$ and  $\frac{d_n}{\theta_n}O\left(\frac{1}{n}\right)$
converge to zero due to the boundedness of $g_n/\theta_n$ and $d_n/\theta_n$,
$\Theta_n\geq c>0$ for large $n$ if and only if  (\ref{crucrit}) is satisfied. Hence
$(U_m, D(T_m)) :=(G_m+D_m+A_m+B_m ,D(A_m+G_m^0+D_m^0))$ generates an analytic semigroup as the positive perturbation of the diagonal operator $(T_m,D(T_m))$. However, by the closedness,
\begin{align*}
(A_m+G_m+D_m+B_m,D(A_m+G_m^0+D_m^0))& = (T_m+G_m^-+D_m^++B_m, D(T_m))\\
& = \overline{(T_m+G_m^-+D_m^++B_m, D(T_m))} \\&= (Y_m, D(Y_m)).
\end{align*}
\end{proof}
\begin{remark}
We note that (\ref{crucrit}) implies that both $\frac{a_n}{\theta_n}$ and $\frac{\Delta_n^{(m)}}{n^m}$ must be bounded away from 0 and thus, in particular, (\ref{riai}) is satisfied so that $(G_{K_m}(t))_{t\geq 0}$ is an analytic and compact semigroup in its own right.
\end{remark}
\section{Asynchronous exponential growth}\label{sec5}
\begin{proposition}
If assumptions (\ref{condi3}) and (\ref{crucrit}) hold, then $(G_{U_m}(t))_{t\ge0}$ is an analytic
and compact semigroup.
\end{proposition}
\begin{proof}
By (\ref{thesame}), $(G_{U_m}(t))_{t\ge0}$ can be considered to be generated as the perturbation
$U_m = T_m + (G_m^-+D^+_m+B_m)$ so, as in the proof of Theorem~\ref{mth1}.4,
$R(\lambda, T_m)X_m\subset D(U_m)$ implies
\[
R(\lambda, U_m) = R(\lambda, T_m)(I-(G_m^-+D^+_m+B_m)R(\lambda, T_m))^{-1},
\]
where, by
\[
(\lambda I-T_m)R(\lambda, U_m) = (I-(G_m^-+D^+_m+B_m)R(\lambda, T_m))^{-1},
\]
the operator $(I-(G_m^-+D^+_m+B_m)R(\lambda, T_m))^{-1}$ is bounded. Hence, $R(\lambda, U_m)$ is
compact, provided $R(\lambda, T_m)$ is compact and that was proved in Theorem~\ref{mth1}.4.
Since $(G_{U_m}(t))_{t\ge0}$ is analytic, its compactness follows from \cite[Theorem II.4.29]{EN}.
\end{proof}

\begin{proposition}
The semigroup $(G_{U_m}(t))_{t\ge0}$ is irreducible.
\end{proposition}
\begin{proof}
By \cite[Proposition 7.6]{Clem} it suffices to show that $R(\lambda, U_m)$ is irreducible for some $\lambda>s(U_m)$.
To simplify the calculations, we use the representation
\[
(U_m, D(T_m)) =  (K_m +D_m^++B_m, D(T_m))=: (K_m +\mathsf{B}_m, D(T_m)),
\]
see \eqref{thesame}, corresponding to \eqref{feco2} and \eqref{betani}.

Then, using the formula for the resolvent from \cite[Proposition 9.29]{BaAr} (compare \eqref{res1})
\begin{equation}
R(\lambda, U_m) f= \sum\limits_{k=0}^\infty R(\lambda, K_m)[{\sf B}_mR(\lambda, K_m)]^k  f, \quad
f\in X_m,\quad \lambda >s(U_m).
\label{res1a}
\end{equation}
we have
\[
R(\lambda, U_m) \geq R(\lambda, K_m) + R(\lambda, K_m){\sf B}_mR(\lambda, K_m)
\]
and, by \eqref{lgf5},
\begin{align*}
u_n&=\bigl([R_n(\lambda, K_m) + R(\lambda, K_m){\sf B}_mR(\lambda, K_m)]f\bigr)_n\\
&=\sum\limits_{i=1}^n\frac{1}{\lambda+\theta_i}
\prod\limits_{j=i}^{n-1}\frac{g_j}{\lambda+ \theta_{j}}\Biggl(f_i+\\
&\phantom{xxx}\left. \sum\limits_{j=i+1}^{\infty}\mathsf{a}_j\mathsf{b}_{i,j}\left(\sum\limits_{s=1}^j
\frac{f_s}{\lambda+\theta_s} \prod\limits_{l=s}^{j-1}\frac{g_l}{\lambda+ \theta_{l}}\right)\right).
\end{align*}
Hence $u_n =0$ if and only if
\begin{align*}
f_i =& 0, \\
\sum\limits_{j=i+1}^{\infty}\mathsf{a}_j\mathsf{b}_{i,j}\left(\sum\limits_{s=1}^j
\frac{f_s}{\lambda+\theta_s} \prod\limits_{l=s}^{j-1}\frac{g_l}{\lambda+ \theta_{l}}\right)=&0,
\end{align*}
for $1\leq i\leq n$. This implies
\[
\mathsf{b}_{i,j} = 0\quad \mathsf{for}\quad 1\leq i\leq n,\quad j\geq n+1.
\]
In particular,
\[
\mathsf{b}_{n,i+1}=0 \quad \text{for}\quad 1\leq n\leq i.
\]
This contradicts (\ref{betain}) that requires
\[
\sum\limits_{n=1}^{i}n\mathsf{b}_{n,i+1} = i+1-\frac{d_{i+1}}{a_{i+1}+d_{i+1}}>0, \qquad n\geq 1.
\]
Hence $R(\lambda, U_m)f>0$ provided $0\neq f \geq 0$ and thus $R(\lambda, U_m)$, and hence
$(G_{U_m}(t))_{t\ge0}$, are irreducible.
\end{proof}

Thus \cite[Theorem VI.3.5]{ENshort} yields the following result.
\begin{theorem}
Assume that  (\ref{condi3}) and \eqref{crucrit} are satisfied. Then there exist a strictly positive $ e \in X_m$,
a strictly positive $ h\in X_m^*, $ $M\geq 1$ and $\epsilon>0$ such that for any ${f}^{in}\in X_n$ and $t\geq 0$
\begin{equation}
\|e^{-s(U_m)t}G_{U_m}(t)f^{in} -\langle h, f^{in}\rangle e\|_{[m]}\leq Me^{-\epsilon t}.
\label{AEG}
\end{equation}\label{thAEG}
\end{theorem}

\section{Examples}\label{sec6}
\begin{example}\label{ex1}
To illustrate the above result, consider the growth-fragmentation problem
\begin{align}
\frac{df_1}{dt}&= -g_1f_1 + \sum\limits_{i=2}^\infty a_{i}b_{1,i}f_{i},\nonumber\\
\frac{df_n}{dt}&= g_{n-1}f_{n-1}-(a_n+g_n)f_n, \quad n\geq 2\nonumber\\
f_n(0)&= f^{in}_n, \quad n\geq 1,
\label{feco1ex}
\end{align}
where
\[
b_{n,i} = \left\{
\begin{array}{lcl}i&\mathrm{for}& n=1,\\
0&\text{otherwise};
\end{array}\right.
\]
that is, any particle breaks down into monomers. We see that
$$
\Delta_n^{(m)} = n^m-\sum\limits_{k=1}^{n-1}k^mb_{k,n} = n^m-n
$$
and hence (\ref{crucrit}) is satisfied for any $m>1$. Since $d_n=0$ for all $n$, we take any unbounded $(a_n)_{n\in\mathbb{N}}$ and $(g_n)_{n \in \mathbb N}$ satisfying
\begin{equation}
\gamma a_n \leq  g_n\leq g a_n, \quad n\geq 2
\label{angnle}
\end{equation}
for some $\gamma\leq g$. Then the semigroup $(G_{U_m}(t))_{t\ge0}$ that solves (\ref{feco1ex}) is analytic and compact in $X_m$
for any $m>1$ and Theorem \ref{thAEG} holds. Moreover, we observe that
\begin{align}
\lambda f_1&= -g_1f_1 +  \sum\limits_{i=2}^\infty a_{i}b_{1,i}f_{i},\nonumber\\
\lambda f_n&= g_{n-1}f_{n-1}-(a_n+g_n)f_n, \quad n\geq 2
\label{feco1ex0}
\end{align}
can be explicitly solved.  Indeed, let $\lambda\geq 0$ and, starting from the second equation,
we get
\[
f_{n,\lambda} = \frac{g_1 f_1}{\lambda + g_n+a_n}\prod\limits_{j=2}^{n-1} \frac{g_j}{\lambda+ g_j+a_j},
\quad n\geq 2
\]
and
\[
\sum\limits_{n=2}^\infty a_{n}b_{1,n}f_{n,\lambda} = g_1\sum\limits_{n=2}^\infty
\frac{a_n n}{\lambda+a_n+g_n}\prod\limits_{j=2}^{n-1} \frac{g_j}{\lambda+ g_j+a_j}.
\]
Now, by (\ref{riai}), $\frac{g_j}{\lambda+ g_j+a_j} \leq c = \frac{g}{1+g}< 1$ and thus,
\begin{equation}
g_1\sum\limits_{n=2}^\infty\frac{a_n n}{\lambda+a_n+g_n}\prod\limits_{j=2}^{n-1}
\frac{g_j}{\lambda+g_j+a_j}\leq g_1\sum\limits_{n=2}^\infty n c^{n-2} <\infty.
\label{unest}
\end{equation}
Hence, after dividing by $f_1\neq 0,$ the first equation  takes the form
\[
\psi(\lambda):=\frac{\lambda +g_1}{g_1} = \sum\limits_{n=2}^\infty\frac{a_n n}{\lambda + a_n+g_n}
\prod\limits_{j=2}^{n-1} \frac{g_j}{\lambda+ g_j+a_j}=:\phi(\lambda).
\]
By (\ref{unest}), the series defining $\phi$ is uniformly convergent on $[0,\infty),$ hence $\phi$ is continuous
there and
\[
\phi(0) = \sum\limits_{n=2}^\infty\frac{a_n n}{g_n}\prod\limits_{j=2}^{n} \frac{g_j}{g_j+a_j}.
\]
Using (\ref{angnle}), we have, for $q = \frac{\gamma}{1+\gamma}$,
\[
\phi(0) \geq \frac{1}{g}\sum\limits_{n=2}^\infty nq^{n-1} = \frac{1}{g} \frac{d}{d q}\sum\limits_{n=2}^\infty q^{n}
= \frac{1}{g}\frac{d}{dq}\frac{q^2}{1-q} = \frac{1}{g} \frac{2q-q^2}{(1-q)^2} = \frac{1}{g}\left(\frac{1}{(1-q)^2}-1\right);
\]
that is,
\[
\phi(0) \ge \frac{(\gamma+1)^2 -1}{g}>1
\]
provided
\begin{equation}
g+1<(\gamma+1)^2 \leq (g+1)^2,
\label{ggamcond}
\end{equation}
where the second inequality follows from $\gamma\leq g$, implied by (\ref{angnle}). We see that, in particular, if
$\gamma =g$; that is, $g_n = ga_n$, \eqref{ggamcond} is satisfied.
Also, $\lim_{\lambda\to \infty}\phi(\lambda) = 0$. On the other hand, $\psi(0) =1$ and
$\lim_{\lambda\to \infty}\psi(\lambda) =+\infty$. Since $\phi$ is decreasing and $\psi$ is increasing,
there is exactly one $\lambda_0>0$ for which (\ref{feco1ex0}) has a solution
(with arbitrary $f_1$ that can be set to 1). Moreover, we see that
\[
\sum\limits_{n=1}^{\infty} a_nn^m f_{n,\lambda} = g_1\sum\limits_{n=2}^\infty
\frac{a_n n^m}{\lambda + a_n+g_n}\prod\limits_{j=2}^{n-1}
\frac{g_j}{\lambda+g_j+a_j}\leq g_1\sum\limits_{n=2}^\infty{ n^m c^{n-2}} <\infty,
\]
and thus $f_{\lambda_0} = (f_{n,\lambda_0})_{n\in \mathbb N}$ is the Perron eigenvector of the generator $U_m$.
\end{example}

\begin{example}
The dominant eigenvalue $\lambda_0$ can be explicitly found in certain cases. Let us consider general problem (\ref{feco1}) with $g_n = r n, d_n = 0$ for all $n\in \mathbb N$  and  some $r>0$ and with other  coefficients satisfying the assumptions of Theorem \ref{thAEG}. Let $f_{\lambda} = (f_{n,\lambda})_{n\in \mathbb N}\in D(U_m)$ satisfy
\begin{align}
\lambda f_{1,\lambda}&= -rf_{1,\lambda} +
\sum\limits_{i=2}^\infty a_{i}b_{1,i}f_{i,\lambda},\nonumber\\
\lambda f_{n,\lambda}&= r(n-1)f_{n-1,\lambda}-(a_n+ rn)f_{n,\lambda}  +
\sum\limits_{i=n+1}^\infty a_{i}b_{n,i}f_{i,\lambda}, \quad n\geq 2.
\label{feco1ex0b}
\end{align}
Multiplying the $n$-th equation by $n$ and summing them, we obtain
\[
\lambda\sum\limits_{n=1}^\infty nf_{n,\lambda} = r\sum\limits_{n=1}^\infty nf_{n,\lambda}.
\]
The above is satisfied if either $\lambda = r$ or $\sum\limits_{n=1}^\infty nf_{n,\lambda}=0$. Since we know that the Perron eigenvector must be positive, we obtain  that $\lambda_0=r$ is the Perron eigenvalue. As a byproduct, we see that any eigenvector $f_\lambda$ belonging to an eigenvalue $\lambda\neq r$ must satisfy $\sum\limits_{n=1}^\infty nf_{n,\lambda}=0$.

To conclude, let us consider the transposed matrix
\[
\mathcal{U}^T = \left(\begin{array}{cccccc}-g_1&g_1&0&0&0&\ldots\\a_2b_{1,2}&-(g_2+a_2)&g_2&0&0&\ldots\\
a_3b_{1,3}&a_3b_{2,3}&-g_3+a_3&g_3&0&\ldots\\
\vdots&\vdots&\vdots&\vdots&\vdots&\vdots\\
a_nb_{1,n}&a_nb_{2,n}&\ldots&-(g_n+a_n)&g_n&\ldots\\
\vdots&\vdots&\vdots&\vdots&\vdots&\vdots
\end{array}
\right).
\]
Let $U_m^*$ be the adjoint to $U_m$ acting in $X_m^* =\{(v_n)_{n\in\mathbb{N}}\;; \sup\limits_{n\in \mathbb N}n^{-m}|v_n|<\infty\}$ and let $f^* \in D(U^*_m)$. Then, by definition
$$
\langle U^*_mf^*, f\rangle = \langle f^*, U_mf\rangle, \quad f \in D(U_m).
$$
Taking $f = (\delta_{n,N})_{n\in \mathbb N}$, we see that
$$
[U^*_mf^*]_N = \langle f^*,U_m f\rangle = \sum\limits_{n=1}^{N-1} f^*_n a_N b_{n, N} - (g_N+a_N)f^*_N + g_Nf^*_{N+1} = [\mathcal U^T f^*]_{N}
$$
hence $U^*_m$ is a restriction of $\mathcal U^T$ to $D(U^*_m) \subset D(U^*_{m,\max}) = \{f \in X_m^*\ :\ \mathcal U^Tf \in X_m^*\}$. On the other hand, let $f^* \in D(U^*_{m,\max})$, $f \in D(U_m)$. Then, since $D(U_m)$ is a weighted $l^1$ space, $\bigcup_{N=1}^\infty P_ND(U_m)$, where $P_N$ is the projection defined in \eqref{proj}, is a core for $U_m$. Using the fact that $U_mP_N D(U_m)$ is finite dimensional, for each $N$
$$
\langle \mathcal U^Tf^*, P_Nf\rangle = \langle f^*, U_mP_Nf\rangle = \langle U^*_mf^*, P_Nf\rangle
$$
and hence, passing to the limit with $N\to \infty$,  $f^*\in D(U^*_m)$. Thus $U^*_m = \mathcal U^T$ with $D(U^*_m) = D(U^*_{m,\max})$.

Using the assumption that $g_n = rn,$ we see that $h = (1,2,\ldots,n,\ldots)\in D(U_m^*)$ for any
$m\geq 1$ and
\[
U^*_m h = rh.
\]
Thus, by Theorem \ref{thAEG},
\[
U_m(t)f^{in} = e^{rt} \left(\sum\limits_{n=1}^{\infty}nf^{in}_n\right) e + O(e^{r' t})
\]
for some $r'<r$, where $e$ is the Perron eigenvector with unit mass; that is $e = f_{\lambda_0}/\sum_{n=1}^\infty nf_{n,\lambda_0}$.
\end{example}

To illustrate the formulas derived in the last two examples, we let 
$m=2$, $r=1$, $a_n = 2n$, $f^{in}_n = \delta_{n,10}10$ and
integrate \eqref{feco1ex} numerically in the time interval $t\in[0,20]$. As evident from Fig.~\ref{fig1},
the solution $f(t)$ very quickly settles to its asymptotic limit $\langle h, f^{in}\rangle e$ (see the top-right diagram),
while in complete agreement with Theorem~\ref{thAEG}, the deviation
$\|e^{-rt}G_{U_m}(t)f^{in} - \langle h, f^{in}\rangle e\|_{[m]}$ decreases exponentially as $t$ increases
(see the bottom-left diagram).

\begin{figure}[ht!]
\includegraphics[scale=0.85]{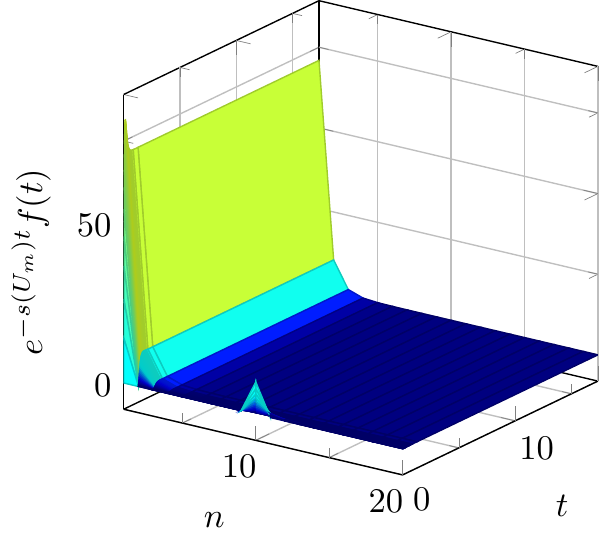}\;\;\; \includegraphics[scale=0.85]{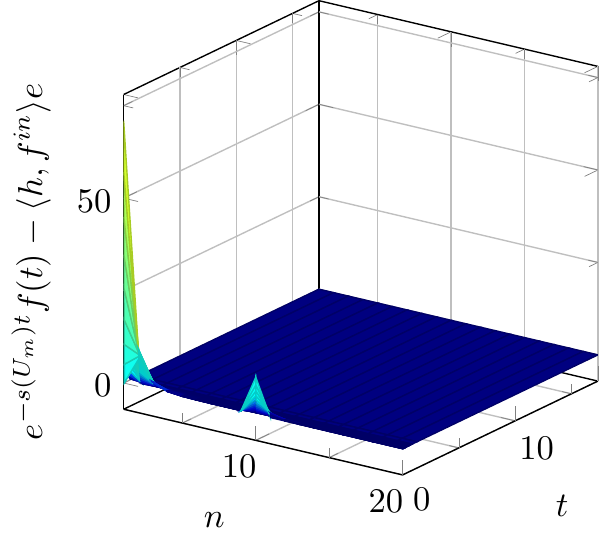} \\
\includegraphics[scale=0.85]{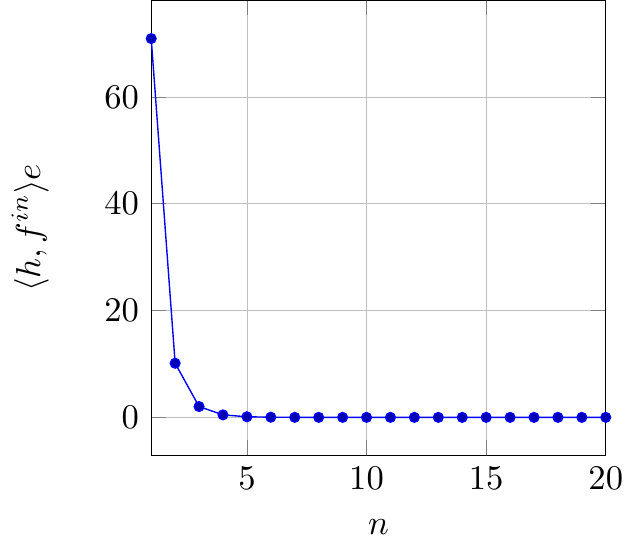}\;\;\; \includegraphics[scale=0.85]{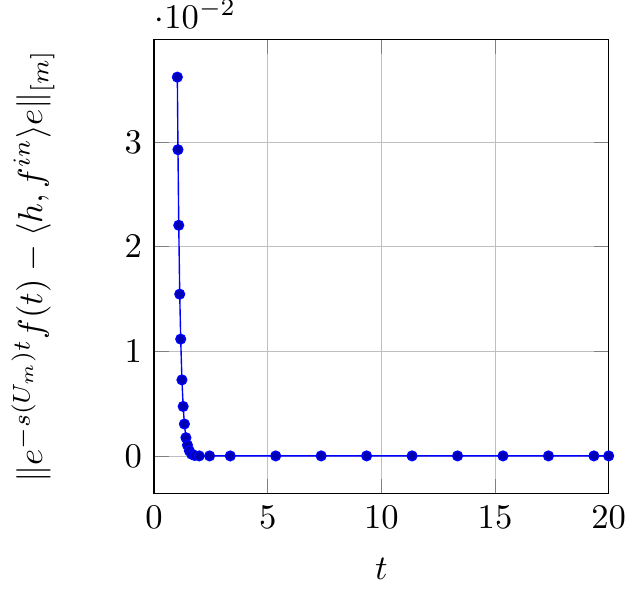}
\caption{The long time behavior of \eqref{feco1ex}.  The semigroup solution
$f(t) = G_{U_m}(t)f^{in}$ of \eqref{feco1ex} (top-left);
the asymptotic error $e^{-s(U_m)t}f(t) - \langle h, f^{in}\rangle e$ (top-right);
the asymptotic mass distribution $\langle h, f^{in}\rangle e$ (bottom-left)
and the evolution of the asymptotic error
$\|e^{-s(U_m)t}f(t) - \langle h, f^{in}\rangle e\|_{[m]}$, for $t\ge 1$ (bottom-right).
}\label{fig1}
\end{figure}

A crucial role in the analysis is played by (\ref{crucrit}). It ensures that most of the mass of the daughter particles
is concentrated in smaller particles, \cite{Bana12b}.  A large class of fragmentation kernels, that can be considered
to be a discrete equivalent of the homogeneous kernels in continuous fragmentation,  satisfying (\ref{crucrit}) is
presented in the next example.

\begin{example}\label{ex2}
Assume that $b_{k,n}$ can be written as
\begin{equation}
b_{k,n} = \zeta(n) h\left(\frac{k}{n}\right), \quad 1\leq k\leq n-1,\quad  n \in \mathbb N,
\label{nh}
\end{equation}
where $h$ is a Riemann integrable function on $[0,1]$ and $\zeta(n)$ is an appropriate sequence
that ensures that (\ref{bin}) is satisfied. By (\ref{bin}), we have
\[
1 = \zeta(n)(n-1) \sum\limits_{k=1}^{n-1} \frac{k}{n}h\left(\frac{k}{n}\right)\frac{1}{n-1}.
\]
Since
\[
\frac{k-1}{n-1}\leq \frac{k}{n}\leq \frac{k}{n-1}
\]
for $1\leq k\leq n$, we have
\[
\lim\limits_{n\to \infty}\sum\limits_{k=1}^n \frac{k}{n}h\left(\frac{k}{n}\right)\frac{1}{n-1} = \int_{0}^{1}zh(z)dz
\]
and thus
\[
\lim\limits_{n\to \infty}(n-1)\zeta(n) = \frac{1}{\int_{0}^{1}zh(z)dz}.
\]
Therefore
\begin{align*}
\lim\limits_{n\to \infty}\sum\limits_{k=1}^{n-1}  \left(\frac{k}{n}\right)^pb_{k,n}
&=  \lim\limits_{n\to \infty}\zeta(n)(n-1)\sum\limits_{k=1}^{n-1}  \left(\frac{k}{n}\right)^ph
\left(\frac{k}{n}\right)\frac{1}{n-1}\\
&=  \frac{\int_{0}^{1}z^ph(z)dz}{\int_{0}^{1}zh(z)dz}<1.
\end{align*}
Thus
\[
\liminf\limits_{n\to \infty}\frac{\Delta_n^{(p)}}{n^p}
= \lim\limits_{n\to \infty}\frac{n^p-\sum\limits_{k=1}^{n-1}k^pb_{k,n}}{n^p} >0
\]
and hence (\ref{crucrit}) is satisfied.

We note that (\ref{nh}) is obviously satisfied by the binary uniform fragmentation
\[
b_{n,i} = \frac{2}{i-1}, \quad n=1,\ldots,i-1.
\]
Another example is offered by the binary fragmentation written in terms of a symmetric infinite matrix
$(\psi_{i,j})_{i,j\geq 1}$ as
\[
\frac{d f_n}{dt} = -\frac{1}{2}f_n\sum\limits_{i=1}^{n-1}\psi_{i,n-i}
+\sum\limits_{i=n+1}^{\infty}\psi_{n,i-n}u_{i},\quad n\geq 1,
\]
see \cite{BaCa90, daCo95, ZM1}. Translating into our notation, we get
\[
b_{n,i} = \frac{\psi_{n,i-n}}{a_i},\quad
a_n = \frac{1}{2}\sum\limits_{i=1}^{n-1}\psi_{i,n-i}, \quad i\geq 2,\quad 1\leq n\leq i-1.
\]
Typical cases in the polymer degradation are
\begin{align*}
\psi_{i,j} &= (i+j)^\beta,\\
\psi_{i,j} &= (ij)^\beta.
\end{align*}
The first case gives $a_n = \frac{1}{2}n^\beta(n-1)$ and $b_{n,i} = \frac{2}{i-1}$ and hence it is a
uniform binary fragmentation (see the long time behavior of $G_{U_m}(t)f^{in}$, with
$m=2$, $\beta=\frac{1}{10}$, $g_n = d_n = n^{1+\beta}$ and $f_n^{in} = \delta_{10,n}10$, in Fig.~\ref{fig2}).
In the second case, we have
\[
b_{n,i} = \frac{n^\beta (i-n)^\beta}{a_i}
= \frac{i^{2\beta}}{a_i} \left(\frac{n}{i}\right)^\beta\left(1-\frac{n}{i}\right)^\beta
\]
and (\ref{nh}) is satisfied with
\[
\zeta(n) =  \frac{n^{2\beta}}{a_n}\quad\text{and}\quad h(z) = z^\beta (1-z)^\beta.
\]
(the typical qualitative behavior of $G_{U_m}(t)f^{in}$, with
$m=2$, $\beta=\frac{1}{10}$, $d_n = g_n = n^{1+\beta}$ and $f_n^{in} = \delta_{10,n}10$, is shown in Fig.~\ref{fig3}).
\end{example}
\begin{example}
On the other hand, the  fragmentation process given by
\begin{align}
&b_{1,2} = 2, \quad\text{and}\quad b_{1,i}=b_{i-1,i} =1, \nonumber\\
&b_{n,i}=0, \quad i\geq 2,\quad 2\leq n\leq i-2,
\label{bbad}
\end{align}
obviously does not satisfy (\ref{crucrit}) and, in fact, the corresponding semigroup is neither analytic,
nor compact, see \cite{Bana12b}.
\end{example}

\begin{figure}[ht!]
\includegraphics[scale=0.85]{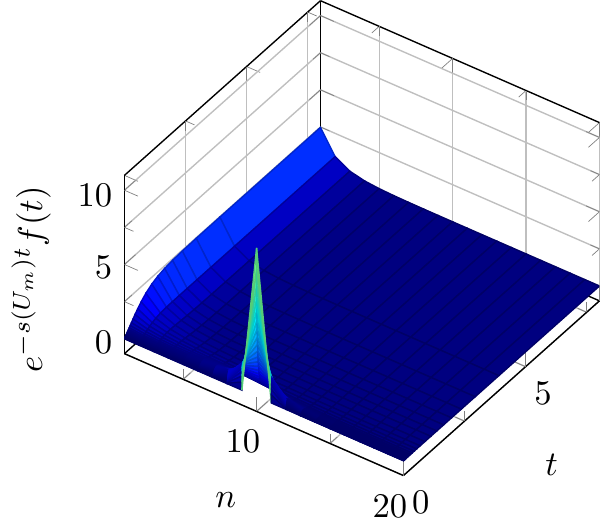}\;\;\; \includegraphics[scale=0.85]{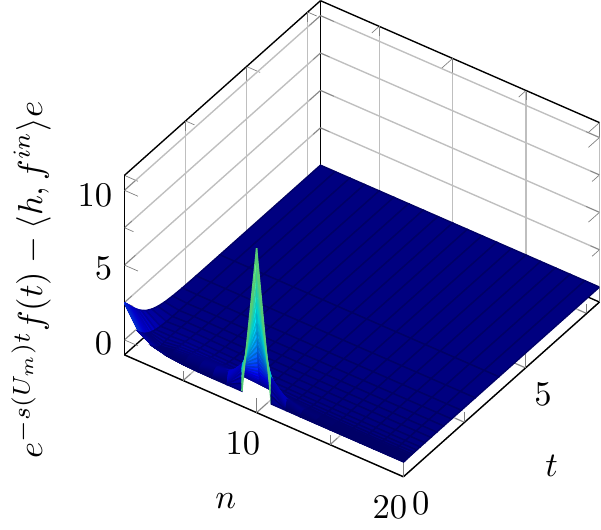} \\
\includegraphics[scale=0.85]{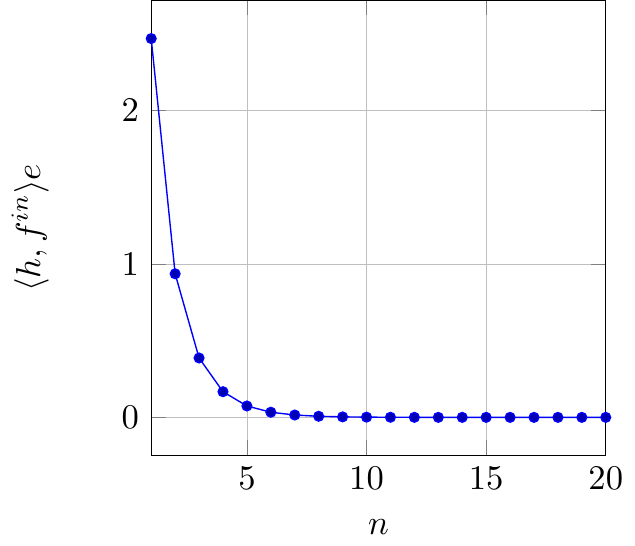}\;\;\; \includegraphics[scale=0.85]{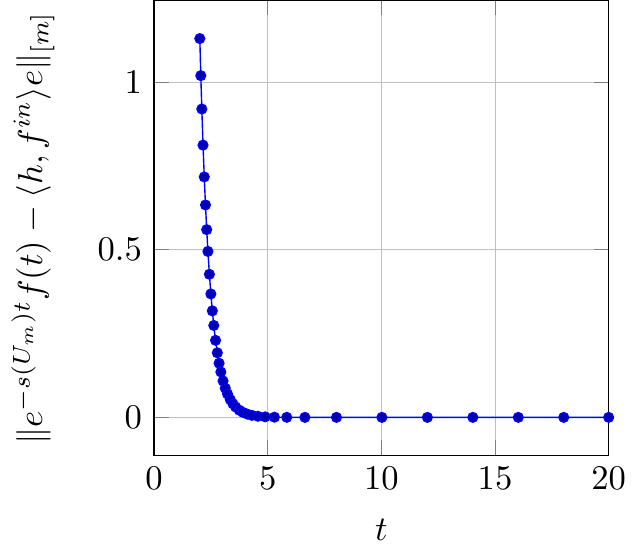}
\caption{The long time behavior of \eqref{feco1}, $\psi_{i,j} = (i+j)^\beta$.  The semigroup solution
$f(t) = G_{U_m}(t)f^{in}$ of \eqref{feco1} (top-left);
the asymptotic error $e^{-s(U_m)t}f(t) - \langle h, f^{in}\rangle e$ (top-right);
the asymptotic mass distribution $\langle h, f^{in}\rangle e$ (bottom-left)
and the evolution of the asymptotic error
$\|e^{-s(U_m)t}f(t) - \langle h, f^{in}\rangle e\|_{[m]}$, for $t\ge 1$ (bottom-right).
}\label{fig2}
\end{figure}

\begin{figure}[ht!]
\includegraphics[scale=0.85]{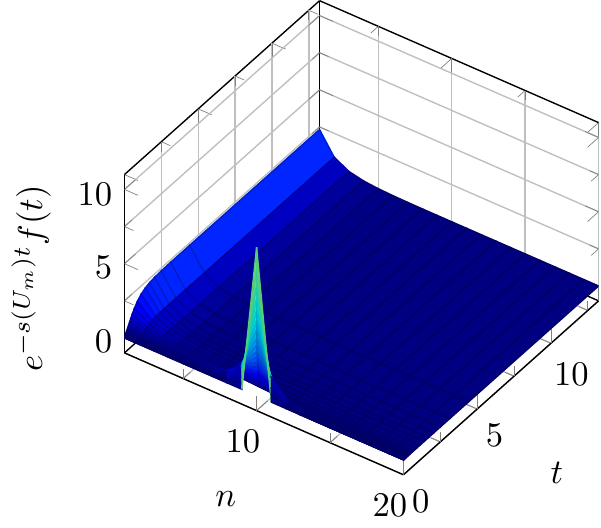}\;\;\; \includegraphics[scale=0.85]{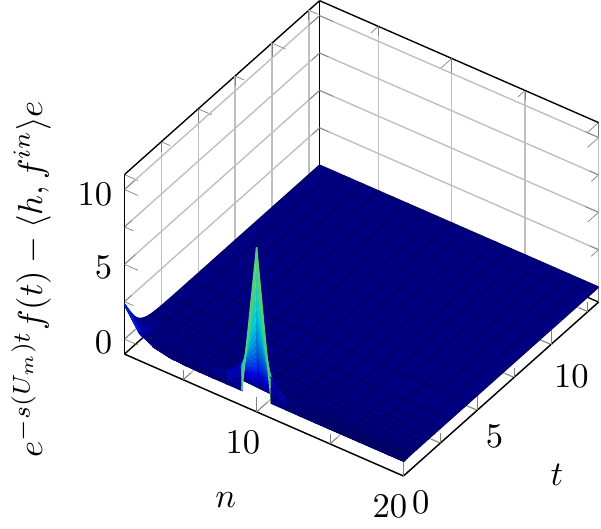} \\
\includegraphics[scale=0.85]{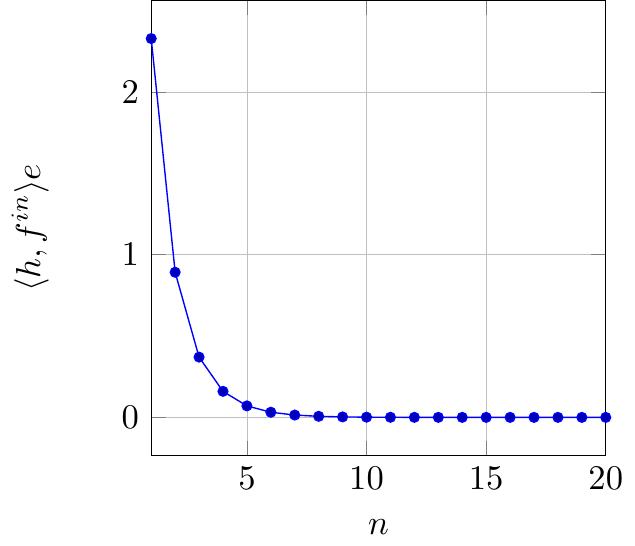}\;\;\; \includegraphics[scale=0.85]{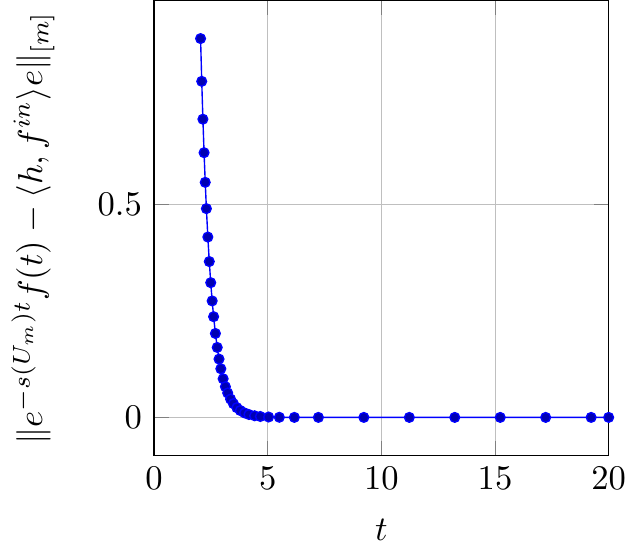}
\caption{The long time behavior of \eqref{feco1}, $\psi_{i,j} = (ij)^\beta$.  The semigroup solution
$f(t) = G_{U_m}(t)f^{in}$ of \eqref{feco1} (top-left);
the asymptotic error $e^{-s(U_m)t}f(t) - \langle h, f^{in}\rangle e$ (top-right);
the asymptotic mass distribution $\langle h, f^{in}\rangle e$ (bottom-left)
and the evolution of the asymptotic error
$\|e^{-s(U_m)t}f(t) - \langle h, f^{in}\rangle e\|_{[m]}$, for $t\ge 1$ (bottom-right).
}\label{fig3}
\end{figure}

\section{Appendix: an alternative view at the model}\label{sec7}
In Theorem \ref{mth1}, we have seen a regularizing role played by the diagonal operator induced by $\mathcal{A}$
even in the case not involving the full fragmentation operator. In many applications, however, (\ref{feco1}) models
a combination of two independent processes -- the birth-and-death process and the fragmentation process and it
is important to investigate when they exist irrespective of each other. In other words, we consider (\ref{feco1}) as
\begin{equation}
\frac{d}{dt}f= \mathcal{G}f+ \mathcal{D}f+ \mathcal{F}f,\qquad
 f(0) = f^{in}.
\label{fecof3}
\end{equation}
The pure birth-and-death problem
\begin{equation}
\frac{d}{dt}f= \mathcal{V}f =\mathcal{G}f+ \mathcal{D}f,\qquad
 f(0) = f^{in}
\label{fecof4}
\end{equation}
has been extensively analysed in the space $X_{0},$ see e.g. \cite[Chapter 7]{BaAr}. Its behaviour in $X_{m}$
creates, however, unexpected challenges. First, we observe
\begin{example}
If there is $C$ such that
\begin{equation}
g_n \leq Cn, \quad n\geq 1,
\label{ass1agrowth}
\end{equation}
then there is a realization of the growth expression $\mathcal{G}$ that generates a $C_0$-semigroup in
$X_{m}$. Indeed, this again follows from the Kato--Voigt theorem. We consider $\mathcal{G}$ as the perturbation of
$\mathcal{G}^0$ by $\mathcal{G}^-$; that is, we introduce $G^0_m = \mathcal{G}^0|_{D(G^0_m)}$, with
\[
D(G^0_m) = \{f \in X_m;\; \mathcal G^0 f \in X_m\}.
\]
Then, as in (\ref{1est}), for $f \in D(G_m^0)$,
\begin{equation}
\sum\limits_{n=1}^\infty n^m[(G^0_m +G^{-}_m)f]_n
= \sum\limits_{n=0}^{\infty} n^m f_n \left(g_n \frac{(n+1)^m-n^m}{n^m} \right)\leq C'\|f\|_{[m]},
\label{1estg}
\end{equation}
for some constant $C'$. Hence, there is an extension of $G^0_m+G^-_m$ generating a $C_0$-semigroup in
$X_{m}$. On the other hand, if for some $c,C>0$
\begin{equation}
cn^q\leq g_n \leq Cn^q, \quad n\geq 1, q>1,
\label{ass1agrowth'}
\end{equation}
then there is no realisation of $\mathcal{G}$ with resolvent bounded in $X_{m}$ with $q \leq  m+1$.
Indeed, the resolvent of the generator, if it exists, must be given by (\ref{lgf5}),
\begin{equation}
\label{lgf5g}
[R_\lambda f]_n = \sum_{i=1}^{n} \frac{f_i}{\lambda + g_n}
\prod_{j=i}^{n-1} \frac{g_{j}}{\lambda + g_{j}}, \quad n\ge 1.
\end{equation}
Let us fix $\lambda$. Then
\[
\prod_{j=i}^{n-1} \frac{g_{j}}{\lambda + g_{j}}\geq g_\lambda
:= \prod_{j=1}^{\infty} \frac{g_{j}}{\lambda + g_{j}},
\]
where $g_\lambda\neq 0$, and, for $f \in X_{m,+}$,
\begin{align*}
\| R_\lambda f \|_{[m]} &= \sum_{n=1}^{\infty} n^m \sum_{i=1}^{n} \frac{f_i}{\lambda + g_n}
\prod_{j=i}^{n-1} \frac{g_{j}}{\lambda + g_{j}}
=\sum_{i=1}^{\infty} f_{i} \sum_{n=i}^{\infty}  \frac{n^m}{\lambda + g_n}
\prod_{j=i}^{n-1} \frac{g_{j}}{\lambda + g_{j}}\\
&\geq   g_\lambda \sum_{i=1}^{\infty} f_{i} \sum_{n=i}^{\infty} \frac{n^m}{(\lambda + g_n)}
\geq g_\lambda C^{-1}\sum_{i=1}^{\infty} f_{i} \sum_{n=i}^{\infty} \frac{1}{n^{q-m}}.
\end{align*}
Hence $R_\lambda$ is not bounded if (\ref{ass1agrowth'}) is  satisfied and hence, in particular, there is no realisation
of $\mathcal{G}$ generating a $C_0$-semigroup in $X_m$. We note that for $q = 2$ and $m=1$ we have a
discrete version of the nonexistence result obtained in \cite[Remark 2]{BOR09}.
\end{example}
Let us return to the full birth-and-death model (\ref{fecof3}). As before, we introduce
$V^0_m + V^1_m: =G_m^0+D_m^0 + G_m^-+D_m^+$ on
\[
D(V^0_m)=\{f \in X_m;\; (\mathcal{G}^0+\mathcal{D}^0)f \in X_m\}.
\]
We have
\begin{theorem}\label{theorem3}
\begin{enumerate}
\item If
\begin{equation}
\limsup\limits_{n\to \infty}\Gamma_n \leq C
\label{bdp1}
\end{equation}
for some constant $C\in \mathbb R$, where
\[
\Gamma_n = {g_n}\left(\left(1+\frac{1}{n}\right)^m-1\right)-{d_n}\left(1-\left(1-\frac{1}{n}\right)^m\right),
\]
then there is an extension $ {V}_m$ of $V^0_m+V^1_m$ that generates a quasicontractive semigroup
$(G_{{V}_m}(t))_{t\ge0}$ on $X_m$.
\item Condition (\ref{bdp1}) is satisfied if either
\subitem a) (\ref{ass1agrowth}) is satisfied, or
\subitem b) $\limsup\limits_{n\to \infty} \dfrac{d_n}{g_n}\geq 1$ and $d_n=O(n^{2})$, or
\subitem c) $\dfrac{d_n}{g_n}\geq 1 + \dfrac{m'-1}{n}$ for sufficiently large $n$ and $m'>m$.
\item If any of the conditions of point 2. is satisfied, then $ {V}_m = \overline{V_m^0+V_m^1}$.
\end{enumerate}
\end{theorem}
\begin{proof}
Statement 1. of the theorem follows in a standard way as an application of the Kato--Voigt theorem. For
$f \in D(V^0_m)_+$ we have
\begin{align*}
&\sum\limits_{n=1}^\infty n^m[(G^0_m+D^0_m+ G^-_m + D^+_m)f]_n \\
&=  \sum\limits_{n=1}^\infty n^m f_n\left(\left(\left(1+\frac{1}{n}\right)^m-1\right)g_n
-d_n\left(1-\left(1-\frac{1}{n}\right)^m\right)\right)=\sum\limits_{n=1}^\infty  n^m f_n\Gamma_n.
\end{align*}
For statement 2. we observe that
\begin{align*}
\left(1+\frac{1}{n}\right)^m-1 &= \frac{m}{n} + \frac{m(m-1)}{2n^2} + O\left(\frac{1}{n^3}\right), \\
 1-\left(1-\frac{1}{n}\right)^m &= \frac{m}{n} - \frac{m(m-1)}{2n^2} + O\left(\frac{1}{n^3}\right).
\end{align*}
Thus, if 2a) is satisfied, then the positive part of $\Gamma_n$ is bounded. If 2b) is satisfied, then
\[
\Gamma_n \leq d_n\left(\frac{m(m-1)}{n^2} + O\left(\frac{1}{n^3}\right)\right)
\]
for sufficiently large $n$ and hence $\Gamma_n$ is bounded from above. Finally, if 2c) is satisfied, then
\begin{align*}
\Gamma_n &\leq g_n\left(\left(\frac{m(m-1)}{n^2} + O\left(\frac{1}{n^3}\right)\right)
- \frac{m'-1}{n}\left(\frac{m}{n} - \frac{m(m-1)}{2n^2} + O\left(\frac{1}{n^3}\right)\right)\right)\\
&= \frac{g_n}{n^2}\left(m(m-m') + O\left(\frac{1}{n}\right)\right)
\end{align*}
and hence $\Gamma_n$ is negative for large $n$ and thus also bounded from above.

To prove the last statement, we use the approach of \cite[Theorem 7.11]{BaAr}, based on the extension technique,
see \cite[Theorem 6.22]{BaAr}. Let $f \in D({V}_m)_+$. Then
\begin{align}
&\sum_{n=1}^{\infty}n^m(-(g_n+d_n)f_n + g_{n-1}f_{n-1} +
d_{n+1}f_{n+1})\nonumber\\
&=\phantom{xx}\sum\limits_{k=1}^\infty  k^m f_k\Gamma_k+ \lim\limits_{n\to \infty}
(-g_nf_n + d_{n+1}f_{n+1})n^m,
\label{arlbd}
\end{align}
where the limit exists. For honesty, it suffices to prove that for any $f \in D({V}_m)_+$
\[
\lim\limits_{n\to \infty} (-g_nf_n + d_{n+1}f_{n+1})n^m \geq 0.
\]
Assume, to the contrary, that for some $0\leq {f} \in D({V}_m)_+$, the limit is negative so that
there exists $b>0$ such that
\begin{equation}
(-g_nf_n + d_{n+1}f_{n+1})n^m \leq -b,
\label{bn}
\end{equation}
for all $n\geq n_0$ with large enough $n_0$. Thus, for $n\geq n_0$ we have
\[
f_n \geq \frac{b}{n^mg_n} + \frac{d_{n+1}}{g_n}f_{n+1}
\]
and, by induction, for arbitrary  $k$
\[
f_{n} \geq  \frac{b}{g_n}\left(
\sum_{i=0}^{k}\frac{1}{(n+i)^m}\prod\limits_{j=1}^{i}\frac{d_{n+j}}{g_{n+j}}\right).
\]
Because $k$ is arbitrary, we obtain
\[
f_{n} \geq \frac{b}{g_n}\left(
\sum_{i=0}^{\infty}\frac{1}{(n+i)^m}\prod\limits_{j=1}^{i}\frac{d_{n+j}}{g_{n+j}}\right),\quad n\geq n_0.
\]
Thus, if
\begin{equation}
\sum_{n=1}^{\infty}\frac{n^m}{g_n}\left(
\sum_{i=0}^{\infty}\frac{1}{(n+i)^m}\prod\limits_{j=1}^{i}\frac{d_{n+j}}{g_{n+j}}\right)=+\infty
\label{as2}
\end{equation}
(where we put  $\prod_{j=1}^{0} \cdot =1$) is satisfied, then
$\sum_{n=0}^{\infty}n^mf_n = +\infty$ which contradicts $f \in D(V_m)_+$.

Now, if (\ref{ass1agrowth}) is satisfied, we have
\[
\sum_{n=1}^{\infty}\frac{n^m}{g_n}\left(
\sum_{i=0}^{\infty}\frac{1}{(n+i)^m}\prod\limits_{j=1}^{i}\frac{d_{n+j}}{g_{n+j}}\right)
\geq \sum_{n=1}^{\infty}\frac{1}{g_n} =+\infty.
\]
Similarly, if assumption 2.b) is satisfied, we have
\[
\sum_{n=1}^{\infty}\frac{n^m}{g_n}\left(
\sum_{i=0}^{\infty}\frac{1}{(n+i)^m}\prod\limits_{j=1}^{i}\frac{d_{n+j}}{g_{n+j}}\right)
\geq \sum_{n=1}^{\infty}\frac{n^m}{g_n}\left(
\sum_{i=n}^{\infty}\frac{1}{i^m}\right) \geq C\sum_{n=1}^{\infty}\frac{1}{n} =+\infty,
\]
where we used the integral estimate for the inner sum. Finally, if 2.c) is satisfied, we can write
\[
\sum_{n=1}^{\infty}\frac{n^m}{g_n}\left(
\sum_{i=0}^{\infty}\frac{1}{(n+i)^m}\prod\limits_{j=1}^{i}\frac{d_{n+j}}{g_{n+j}}\right)
\geq \sum_{n=1}^{\infty}\frac{n^m}{g_n}\left(
\sum_{i=0}^{\infty}\frac{1}{(n+i)^m}\prod\limits_{j=1}^{i}\left(1+\frac{m'-1}{n+j}\right)\right).
\]
Now, as in the proof of Theorem \ref{mth1},  by the Stirling formula,
\[
\prod\limits_{j=1}^{i}\left(1+\frac{m'-1}{n+j}\right)
= \frac{\Gamma (n+i+m') \Gamma (n+1)}{\Gamma(n+m')\Gamma(n+i+1)}
= O\left(\left(\frac{n+i+m'}{n+m'}\right)^{m'-1}\right)
\]
and we see that the inner series diverges if the second condition of  2.c) is satisfied.
\end{proof}

By \cite[Theorem 2.1]{Bana12b},  under standard assumptions on the fragmentation coefficients $F_m = \overline{A_m+B_m}$ generates a
quasicontractive  $(G_{{F}_m}(t))_{t\ge0}$.
\begin{theorem}
Assume the conditions of Theorem \ref{theorem2}, item 1. and of Theorem \ref{theorem3}, item 3. are satisfied.
Then $Y_m = \overline{V_m+F_m}$ and
\begin{equation}
G_{Y_m}(t)f = \lim\limits_{n\to \infty} \left(G_{  V_m}
\left(\frac{t}{n}\right)G_{F_m}\left(\frac{t}{n}\right)\right)^nf, \quad f \in X_m,\label{TKrep}
\end{equation}
uniformly on bounded time intervals. \label{TrotKat}
\end{theorem}
\begin{proof}
First, we observe that $D(V_m)\cap D( F_m) \supset D(G^0_m)\cap D(D^0_m)\cap D(A^0_m)$
and the latter is dense in $X_m$. Next, we see that
\begin{align*}
[\lambda I - ( V_m +  F_m)]D( V_m)\cap D(A_m)&
\supset [\lambda I- (V_m + F_m)]D(G^0_m)\cap D(D^0_m)\cap D( A_m)\\
& = [\lambda I- (T_m+ G_m^-+D^+_m +B_m) ]D(T_m).
\end{align*}
Since $\overline{(T_m+ G_m^-+D^+_m +B_m, D(T_m))}$ is the generator a semigroup, $[\lambda I- (T_m+ G_m^-+D^+_m +B_m) ]D(T_m)$ is dense in $X_m$ for sufficiently large $\lambda$. Indeed, if $f \in X_m$, then $f = (\lambda I- \overline{T_m+ G_m^-+D^+_m +B_m})u$ for some $u \in D(\overline{T_m+ G_m^-+D^+_m +B_m})$ and $u = \lim_{n\to \infty} u_n$ with $u_n \in D(T_m)$ and $\lim_{n\to \infty}(T_m+ G_m^-+D^+_m +B_m)u_n = \overline{T_m+ G_m^-+D^+_m +B_m}u$. But then $f = \lim_{n\to \infty} (\lambda u_n - (T_m+ G_m^-+D^+_m +B_m)u_n)$; that is, $f \in \overline{[\lambda I- (T_m+ G_m^-+D^+_m +B_m) ]D(T_m)}$.

Since both $(G_{ V_m}(t))_{t\ge0}$ and
$(G_{F_m}(t))_{t\ge0}$ are quasicontractive,  \cite[Corollary 3.5.5]{Pa} implies that
$\overline{ V_m+F_m}$ is the generator of a quasicontractive semigroup. Now
\begin{align*}
\lambda I - Y_m
&= \lambda I -\overline {T_m+ G_m^-+D^+_m +B_m} = \lambda I -\overline {(D^0_m+D^+_m+ G_m^-+G^0_m)+(A_m +B_m)}\\
&\subset \lambda I- \overline{ V_m+F_m}
\end{align*}
and, since both $Y_m$ and $\overline{ V_m+F_m}$ are generators, we must have
$Y_m =\overline{ V_m+F_m}$.
Then (\ref{TKrep}) follows from \cite[Corollary 3.5.5]{Pa}.
\end{proof}

\bibliographystyle{plain}
\bibliography{BLL_Bk}

\end{document}